\documentclass[review]{siamart220329}   

\usepackage{siampaper}

\usepackage[hyperpageref]{backref}
\renewcommand*{\backref}[1]{}
\renewcommand*{\backrefalt}[4]{%
\ifcase #1 %
No citations.%
\or
\marginpar{\tiny cited on p. #2}%
\else
\marginpar{\tiny cited on pp. #2}%
\fi
}

\makeatletter
\AtBeginDocument{
  \def\refstepcounter@optarg[#1]#2{%
    \cref@old@refstepcounter{#2}%
    \if\relax\detokenize{#1}\relax
    \else
    \cref@constructprefix{#2}{\cref@result}%
    \@ifundefined{cref@#1@alias}%
    {\def\@tempa{#1}}%
    {\def\@tempa{\csname cref@#1@alias\endcsname}}%
    \protected@edef\cref@currentlabel{%
      [\@tempa][\arabic{#2}][\cref@result]%
    \csname p@#2\endcsname\csname the#2\endcsname}%
\fi}}
\makeatother

\crefname{algorithm}{Algorithm}{Algorithms}
\Crefname{algorithm}{Algorithm}{Algorithms}

\newcommand{\PCG}{^{\text{\tiny PCG}}}
\newcommand{\BFGS}{^{\text{\tiny BFGS}}}
\newcommand{\DFP}{^{\text{\tiny DFP}}}
\newcommand{\SRI}{^{\text{\tiny SR1}}}
\newcommand{\FOM}{^{\text{\tiny FOM}}}
\newcommand{\DIOM}{^{\text{\tiny DIOM}}}

\newcommand{\cahiernumber}{26}  

\title{
  Quasi-Newton and Krylov Methods for the Solution of Nonconvex Trust-Region Subproblems%
}

\author{
  Johann Bourhis%
  \thanks{%
    National Institute of Applied Sciences of Rennes, France.
    E-mail: \mailto{johann.bourhis@insa-rennes.fr}
  }
  \and
  Oihan Cordelier%
  \thanks{%
    GERAD and Department of Mathematics and Industrial Engineering,
    \'Ecole Polytechnique, Montr\'eal, QC, Canada.
    E-mail: \mailto{oihan.cordelier@polymtl.ca}.
  }
  \and
  Jean-Pierre Dussault%
  \thanks{%
    Department of Mathematics,
    Universit\'e de Sherbrooke, QC, Canada.
    E-mail: \mailto{jean-Pierre.Dussault@USherbrooke.ca}.
    Research partially supported by an NSERC Discovery Grant.
  }
  \and
  Oussama~Mouhtal%
  \thanks{%
    GERAD and Department of Mathematics and Industrial Engineering,
    \'Ecole Polytechnique, Montr\'eal, QC, Canada.
    E-mail: \mailto{oussama-2.mouhtal@polymtl.ca}.
  }
  \and
  Dominique Orban%
  \thanks{%
    GERAD and Department of Mathematics and Industrial Engineering,
    \'Ecole Polytechnique, Montr\'eal, QC, Canada.
    E-mail: \mailto{dominique.orban@gerad.ca}.
    Research partially supported by an NSERC Discovery Grant.
  }
}
\date{\today}

\begin{document}

\maketitle

\thispagestyle{firstpage}
\pagestyle{myheadings}

\begin{abstract}
  We study the solution of symmetric positive-definite linear systems by way of families of full- and limited-memory methods.
  Our contributions are threefold.
  We first derive new relationships between the conjugate-gradient method (CG) and quasi-Newton methods of the Broyden class that refine existing results, and clarify when those methods generate the same iterates and enjoy quadratic termination.
  We extend this perspective to the limited-memory BFGS (LBFGS) method.
  Next, we examine how DIOM, a limited-memory variant of the full orthogonalization Krylov method (FOM), is akin to LBFGS in that it provides a memory lever that is critical in practical performance.
  Finally, we generalize LBFGS and DIOM to the computation of trust-region steps for unconstrained, potentially nonconvex, optimization.
  We report numerical experience on positive-definite linear systems and unconstrained optimization problems.
  The results show that memory is a key algorithmic lever: LBFGS and DIOM are consistently more robust than CG and often achieve comparable accuracy with fewer Hessian-vector products.
  They emerge as viable alternatives to CG when high accuracy is desirable or when operations with the Hessian are at a premium.
  The limited-memory SR1 (LSR1) method can be competitive in full-memory form, but its limited-memory variant suffers from discarded curvature information.
\end{abstract}

\begin{keywords}
  Broyden Class,
  Quasi-Newton Method,
  Conjugate-Gradient Method,
  CG,
  Truncated Conjugate-Gradient Method,
  Limited-Memory quasi-Newton Method,
  LBFGS,
  LSR1,
  Full Orthogonalization Method,
  FOM,
  Direct Incomplete Orthogonalization Method,
  DIOM,
  Truncated Direct Incomplete Orthogonalization Method,
  Trust-Region Algorithm
\end{keywords}

\begin{AMS}
  15A06,  
  65F10,  
  65F22,  
  65F25,  
  90C06,  
  90C20,  
  90C30,  
  90C53,  
\end{AMS}

\section{Introduction}
\phantomsection

We consider the unconstrained optimization problem
\begin{equation}
  \label{P}
  \minimize{z \in \R^n} f(z),
\end{equation}
where $f$ : $\R^n \rightarrow \R$ is twice continuously differentiable.
Newton's method globalized by way of a line-search or trust-region mechanism is among the most widely used techniques to identify a stationary point of~\eqref{P}, and computes a step by approximately minimizing a quadratic model of \(f\) about the current iterate.
A popular way to compute a trust-region step is to employ the truncated conjugate-gradient (CG) method of \citet{steihaug-1983}---a variant of the original method of \citet{hestenes-stiefel-1952} with safeguards to take the trust region and potential negative curvature into account.
A similar procedure, also based on CG, exists for line-search methods \citep{dembo-steihaug-1983}.
One way to derive CG is as a three-term recurrence method based on the \citet{lanczos-1950} process.
CG may suffer from loss of orthogonality, and, though theory predicts that it terminates in at most \(n\) iterations, may require many more to converge in practice.
We exploit relationships between CG and quasi-Newton methods of the \citet{broyden-1970} class for optimization, on the one hand, and between CG and Krylov methods based on the \citet{arnoldi-1951} process on the other, to devise iterative methods that are less sensitive to loss of orthogonality.

Specifically, we investigate the use of quasi-Newton, limited-memory quasi-Newton methods and limited-memory Krylov methods to approximately minimize the quadratic
\begin{equation}\label{quadratic}
  q(x) = \tfrac{1}{2} x^T A x - b^T x + c,
\end{equation}
where \(A = A^T \in \R^{n\times n}\), $b \in \R^{n}$ and $c\in \R$.
When $A$ is positive definite (PD), $q$ is strictly convex and~\eqref{quadratic} admits a unique solution.
CG computes the minimum of a strictly convex quadratic function in at most $n$ steps in exact artithmetic \citep{hestenes-stiefel-1952}, a property called quadratic termination.
We provide new relationships between CG and quasi-Newton methods of the Broyden class applied to the minimization of a strictly convex quadratic that generalize the well-known statement of \citet{broyden-1970}: Broyden-class methods generate the same iterates as CG whenever the quasi-Newton update is positive definite and consequently enjoy quadratic termination.

We focus our numerical study on two specific methods: BFGS, one of the most widely-used Broyden-class quasi-Newton methods, and the full-orthogonalization method (FOM), which is similar in spirit to CG, but is based on the \citet{arnoldi-1951} process.
Both have large memory requirements.
Thus, for practical purposes, we investigate limited-memory variants.
We establish that the limited-memory BFGS method (LBFGS) also generates the same iterates as CG in exact arithmetic when applied to a strictly convex quadratic, as does the direct incomplete orthogonalization method (DIOM), a limited-memory variant of FOM.

We report numerical results on ill-conditioned PD systems that show the advantage of LBFGS and DIOM over CG in the presence of loss of orthogonality. We compare LSR1 with CG on the same systems and highlight the severe limitations of LSR1 when memory is low.

We generalize our findings to potentially indefinite systems by devising truncated variants of LBFGS and DIOM, i.e., variants that account for a trust-region constraint; a constraint that confines the iterates to a Euclidean ball centered at the origin.
We establish that the truncated variants are appropriate to compute a step as part of a trust-region method for unconstrained, potentially nonconvex, optimization.

Our experiments on nonlinear optimization problems, namely data-assimilation and binary classification problems, further highlight the benefits of limited-memory variants. In particular, LBFGS and DIOM often require fewer Hessian-vector products than CG and, also reduce overall runtime.

\subsection*{Related research}

\citet{dixon-1972} showed that, for any differentiable \(f: \R^n \to \R\), any two methods in the Broyden class generate parallel descent directions when using a \emph{perfect}\footnote{A line search is perfect if it is exact and guaranteed to locate the nearest local minimum along the search direction.} line search, provided that the Hessian approximation remains nonsingular.
Consequently, they all generate the same iterates.

If, in addition, the objective is a strictly convex quadratic, the results of \citet{dixon-1972} can be strengthened in several ways.
We summarize a few of them relevant here. 
\citet{broyden-1970} showed that these quasi-Newton methods generate the same iterates as the conjugate gradient method of \citet{hestenes-stiefel-1952}.
\citet[\S\(2\)]{nazareth-1979} showed that the BFGS method with exact line search not only generates the same iterates as CG, but also the same search directions and step lengths.
\citet[Proposition~\(1\)]{forsgren-odland-2018} provide necessary and sufficient conditions for a quasi-Newton method with exact line search, not necessarily of the Broyden type and not necessarily based on the secant equation, to generate search directions parallel to those of CG.

Attention also turned to limited-memory methods.
\citet{shanno-1978} showed that, under an exact line search, nonlinear CG can be interpreted as a memoryless BFGS update in the sense that if the Hessian approximation is reset to the identity at each iteration, both methods generate the same iterates.
His results specialize to the case of a strictly convex quadratic with exact line search, and show that the CG iterates coincide with those of the memoryless BFGS method.
\citet{ek-forsgren-2021} introduce a class of limited-memory quasi-Newton Hessian approximations that generate search directions parallel to those of CG and BFGS.
However, their results do not appear to cover the limited-memory BFGS method.

\subsection*{Contributions}

Our main message is that numerous quasi-Newton methods and certain Krylov methods coincide with CG in exact, but not in floating-point, arithmetic.
Using those methods, it is possible to increase the robustness and efficiency of CG in exchange for extra storage.
Specifically,

\begin{enumerate}
  \item As long as the CG iterations are well defined, and a Broyden method does not encounter zero curvature and does not generate a singular approximation, we provide a recurrence for the coefficient of proportionality between its search direction and that of CG (\Cref{theorem-1}).
  The result includes methods that generate potentially indefinite approximations, even if \(A\) is positive definite, and, in particular, the symmetric rank~\(1\) method.
  The result generalizes one of \citet{broyden-1970}, who restricted the Broyden parameter to nonnegative values.
  \item A characterization of the coefficients of proportionality and step lengths in methods of the \emph{convex} Broyden class as a function of the Broyden parameter (\Cref{prop:optimal-step-lenghts}).
  This result shows how step lengths generated by a method in the convex class behave in an orderly manner, whereas those generated by a method that is not in the convex class, such as SR1, can behave erratically.
  \item As long as the CG iterations are well defined, the directions and step lengths generated by the limited-memory BFGS method (LBFGS) with \emph{any} memory coincide with those generated by CG.
  This result was previously thought to hold only for memory \(1\) \citep{shanno-1978}.
  \item We describe how the FOM Krylov method based on the \citet{arnoldi-1951} process also coincides with CG when applied to~\eqref{quadratic} for as long as the CG iterates are well defined.
  FOM has a limited-memory variant named DIOM, which also coincide with CG for \emph{any} value of the memory parameter.
  \item We describe how nonpositive curvature can be detected efficiently in LBFGS and DIOM so those methods can be used to solve trust-region subproblems for potentially nonconvex optimization.
  \item We provide accompanying implementations of all the methods described.
\end{enumerate}

\subsection*{Notation}

Matrices are denoted by uppercase Latin letters, vectors by lowercase Latin letters and scalars by Greek letters.
By convention, all vectors are column vectors.
Throughout, \(x_k\) denotes an iterate of a method used to minimize~\eqref{quadratic}, $s_k := x_{k+1} - x_k$ is the difference between two consecutive iterates, $g_k := \nabla q(x_k) = A x_k - b$, and $y_k := g_{k+1} - g_k = A s_k.$ to denote the difference between two consecutive gradients.
If \(A = A^T \in \R^{n \times n}\), we write \(A \succ 0\) to indicate that \(A\) is positive  definite (PD).
We use the shorthand SPD for \emph{symmetric and positive definite}.
The $i$th component of a vector $z \in \R^n$, for $i=1,\ldots,n$, is denoted $(z)_i$.
All results below are derived under the assumption of exact arithmetic.

\section{Background}

\subsection{Exact line search methods for quadratic functions}
\phantomsection  

If $d_k$ is a descent direction for $q$ in~\eqref{quadratic} from $x_k$, i.e., $g_k^T d_k < 0$, performing an exact line search consists in finding a step length
\begin{equation}\label{linesearch}
  \alpha_k \in \argmin{\alpha \geq 0} \ q(x_k + \alpha d_k).
\end{equation}
The current iterate $x_k$ is subsequently updated according to
\[
    s_k = \alpha_k d_k, \quad x_{k+1} = x_k + s_k.
\]
Since $q$ is quadratic~\eqref{quadratic}, if $A$ is SPD, the unique solution to~\eqref{linesearch} is
\begin{equation}\label{exact-linesearch}
  \alpha_k = -g_k^T d_k / d_k^T A d_k.
\end{equation}
Note that if $A$ is negative definite,~\eqref{exact-linesearch} yields the global maximum of $q$ from $x_k$ in the direction $d_k$, while if $A$ is indefinite and $d_k^T A d_k \neq 0$,~\eqref{exact-linesearch} yields a stationary point.

\subsection{Methods of the Broyden class}

Secant quasi-Newton methods generate descent directions by maintaining an approximation \(H_k = H_k^T\) such that  \(H_k^{-1} \approx \nabla^2 q(x_k) = A\) that is updated to satisfy the secant equation
\begin{equation}
  \label{eq:secant}
  H_{k+1} y_k = s_k.
\end{equation}
Whenever \(s_k^T y_k \neq 0\) and \(y_k^T H_k y_k \neq 0\), the Broyden family of updates is given by
\begin{subequations}\label{broyden-update}
  \label{broyden-family}
  \begin{align}
    H_{k+1} & = H_k + \frac{s_k s_k^T}{s_k^T y_k} - \frac{H_k y_k y_k^T H_k}{y_k^T H_k y_k} + \phi_k \, (y_k^T H_k y_k) (v_k v_k^T),
    \\
    v_k & = \frac{s_k}{s_k^T y_k} - \frac{H_k y_k}{y_k^T H_k y_k},
    \label{vk}
  \end{align}
\end{subequations}
where $\phi_k \in \R$ is a parameter~\citep{denis-more-1977}.  
If we let $\rho_k := 1 / s_k^T y_k$, the BFGS \citep{broyden1970convergence, fletcher1970new, goldfarb1970family, shanno1970conditioning}
and the DFP \citep{davidon1991variable, fletcher1963rapidly}
formulae result, respectively, from $\phi_k = 1$ and $\phi_k = 0$:
\begin{align}\label{BFGS}
  H_{k+1}\BFGS &= (I-\rho_k s_k y_k^T)H_k(I-\rho_k y_k s_k^T) + \rho_k s_k s_k^T,\\
  H_{k+1}\DFP &= H_k - \frac{H_k y_k y_k^T H_k}{y_k^T H_k y_k} + \rho_k s_k s_k^T.
\end{align}
Assuming that \((s_k - H_k y_k)^T y_k \neq 0\), the choice
\begin{equation}\label{phi-SR1}
  \phi_k\SRI = \frac{y_k^T s_k}{(s_k - H_k y_k)^T y_k},
\end{equation}
yields the SR1 update~\citep{dennis1974characterization}
\begin{equation}\label{SR1}
  H_{k+1}\SRI = H_k + \frac{(s_k - H_k y_k)(s_k - H_k y_k)^T}{(s_k - H_k y_k)^T y_k}.
\end{equation}

To avoid a proliferation of superscripts, we write \(H_{k+1}\) without explicit dependency on \(\phi_k\) unless we refer to \(H_{k+1}\BFGS\), \(H_{k+1}\DFP\) or \(H_{k+1}\SRI\).
However, it will be necessary to keep track of \(\phi_k\) when we talk about search directions.
Hence, given an approximation \(H_k\) and a value of \(\phi_{k-1}\), the quasi-Newton search direction is
\begin{equation}\label{direction}
  d_k^{\phi_{k-1}} = -H_k g_k.
\end{equation}

The following result states that~\eqref{broyden-family} is almost always nonsingular.

\begin{proposition}[\protect{\citealp[Lemma~4.2]{denis-more-1977}}]\label{proposition-1}
  Assume that $H_k$ is nonsingular, $s_k^T y_k \neq 0$, and $y_k^T H_k y_k \neq 0$.
  Let $B_k = H_k^{-1}$.
  Then, $H_{k+1}$ defined by~\eqref{broyden-family} is singular if and only if $\phi_k = \phi_k^c$ where
  \begin{equation}\label{phikc}
    \phi_k^c = \frac{(y_k^T s_k)^2}{(y_k^T s_k)^2 - (y_k^T H_k y_k)(s_k^T B_k s_k)}.
  \end{equation}
\end{proposition}


\begin{proposition}[\protect{\citealp[p.~151]{nocedal-wright-2006}}]\label{proposition-2}
  Let $H_k = H_k^T \succ 0$ and $s_k^T y_k > 0$.
  If $\phi_k > \phi_k^c$, then $H_{k+1} = H_{k+1}^T \succ 0$.
  If $\phi_k < \phi_k^c$, $H_{k+1}$ is nonsingular but may be indefinite.
\end{proposition}

It is necessary to initialize all quasi-Newton methods with an approximation \(H_0 = H_0^T\) at iteration \(k = 0\).
If the update used has hereditary positive definiteness, such as the updates of \Cref{proposition-2} with $\phi_k > \phi_k^c$, it is common to pick \(H_0 \succ 0\), and typically as a multiple of the identity.

\subsection{The preconditioned conjugate gradient method}

Whenever \(A \succ 0\) in~\eqref{quadratic}, CG generates iterates by performing an exact line search along $A$-conjugate descent directions $\{d_k\PCG\}$, i.e., such that \(d_i^T A d_j = 0\) for \(i \neq j\).
If a preconditioner $H_0 = H_0^T \succ 0$ is available, the $k$th preconditioned CG (PCG) descent direction is
\begin{equation}\label{CG}
  d_0\PCG = -H_0 g_0, \qquad
  d_k\PCG = -H_0 g_k + \dfrac{g_k^T H_0 g_k}{g_{k-1}^T H_0 g_{k-1}} d_{k-1}\PCG \quad (k \geq 1).
\end{equation}

Thus, the $k$th iterate of PCG is generated by performing an exact line search along these directions. The algorithm below describe this method.

\begin{algorithm}[ht]
  \caption{PCG}
  \label{Algo:PCG}
  \begin{algorithmic}[1]
    \State $g_0 = \nabla q(x_0)$, $z_0 = -H_0 g_0$, ${\rho}_0 = -g_0^T z_0$, $d_0\PCG = z_0$
    \For{$k = 0, 1, \ldots$}
    \State ${\alpha}_k\PCG = {\rho}_k / (d_k\PCG)^T A d_k\PCG$
    \State $x_{k+1}\PCG = x_k\PCG + {\alpha}_k\PCG d_k\PCG$
    \State $g_{k+1} = g_k + {\alpha}_k\PCG A d_k\PCG$ \Comment{\(= \nabla q(x_{k+1})\)}
    \State $z_{k+1} = -H_0 g_{k+1}$
    \State ${\rho}_{k+1}\PCG = -g_{k+1}^T z_{k+1}$
    \State ${\beta}_k = {\rho}_{k+1} / {\rho}_k$
    \State $d_{k+1}\PCG = z_{k+1} + {\beta}_{k+1} d_k\PCG$
    \EndFor
  \end{algorithmic}
\end{algorithm}

CG is a special case of PCG with $H_0 = I$. In the following, we present the main results for PCG, which will be used throughout this paper. We assume that $q$ in~\eqref{quadratic} is strictly convex and that $H_0 = H_0^T \succ 0$.

\begin{proposition}
  Let $k$ be such that $g_k \neq 0$.
  The gradients calculated by PCG satisfy,
  \begin{subequations}
  \begin{align}\label{gradortho}
    g_k^T H_0 g_i & = 0 \quad (i < k), \\
  \label{direcorhto}
    g_k^T d_i\PCG & = 0 \quad (i < k).
  \end{align}
  \end{subequations}
\end{proposition}

\begin{proof}
  The equivalence between applying PCG to the first-order optimality system $A\,x=b$ and applying CG to the system $H_0^{1/2} A H_0^{1/2} x = H_0^{1/2} b$~\citep[Chap.~11]{GoVa13} allows us to invoke the standard results of~\citep[Theorem~5:1]{hestenes-stiefel-1952}.
\end{proof}

\begin{proposition}
  The directions generated by PCG with preconditioner $H_0 = H_0^T \succ 0$ are $A$-conjugate.
  Consequently, PCG finds the unique minimum of $q$ in at most $n$ steps.
\end{proposition}

\begin{proof}
  The same equivalence with standard CG used in the previous proof yields this result directly from~\citep[Theorem~5:2]{hestenes-stiefel-1952}.
\end{proof}

If $q$ is a nonconvex quadratic,~\eqref{exact-linesearch} corresponds to the step to a stationary point of \(q\) from \(x_k\) in the direction \(d_k\). 
If $d_k^T A d_k \neq 0$ for all \(k\), quadratic termination still takes place and CG finds the unique critical point of \(q\) in at most $n$ steps.
Such is the case if \(A\) is quasi-definite~\citep{vanderbei-1995} for certain linear terms \(b\)~\citep{orban-arioli-2017}.

\section{Relationship between the two methods}

In the next few sections, we require one or more of the following assumptions.

\begin{assumption}%
  \label{asm:q-strictly-convex}
  In~\eqref{quadratic}, \(A = A^T \succ 0\).
\end{assumption}

\begin{assumption}%
  \label{asm:broyden}
  Methods of the Broyden class are applied with exact line search and a given initial matrix \(H_0 = H_0^T \succ 0\).
\end{assumption}

\begin{assumption}%
  \label{asm:pcg}
  The PCG method is applied with the same preconditioner \(H_0 = H_0^T \succ 0\) as the initial matrix used in methods of the Broyden class.
\end{assumption}

The following result generalizes existing results by providing a recurrence formula for the proportionality factor between a search direction generated by a Broyden method and that generated by CG.
It expands and completes \citet[\S\(3\) and Theorem~\(1\)]{broyden-1970}, who restricts \(\phi \geq 0\).
The proof follows the pattern of \citet[\S2]{nazareth-1979}.

\begin{theorem}\label{theorem-1}
  Let \Cref{asm:broyden,asm:pcg} be satisfied, but \(A\) is not necessarily PD.
  Let $x_0$ be any starting point, and assume the PCG iterations \(0, \ldots, j\) are well defined.
  Assume that for \(k = 0, \ldots, j\), $\phi_k \neq \phi_k^c$ and
  \begin{align}
    s_k^T A s_k \neq  0,&\label{cond1}\\
    y_k^T H_k y_k \neq  0.&\label{cond2}
  \end{align}
  Then, as long as $g_k \neq 0$, there exists $\gamma_k^{\phi_{k-1}} \neq 0$ such that $d_k^{\phi_{k-1}} = \gamma_k^{\phi_{k-1}} d_k\PCG$, where $d_k\PCG$ is the $k$th conjugate gradient direction preconditioned with $H_0$.
  In addition,
  \begin{equation}\label{EQUATION}
    d_{k+1}^{\phi_k} =
    \frac{
      \gamma_k^{\phi_{k-1}} g_k^T H_0 g_k + \phi_k g_{k+1}^T H_0 g_{k+1}
    }{
      \gamma_k^{\phi_{k-1}} g_k^T H_0 g_k + g_{k+1}^T H_0 g_{k+1}
    }
    d_{k+1}\PCG = \gamma_{k+1}^{\phi_k} d_{k+1}\PCG.
  \end{equation}
  As a result, the iterates $s_k$ generated by any method in the Broyden family are independent of $\phi_k$ and coincide with the PCG iterates.
\end{theorem}

\begin{proof}
  We establish by induction that the vectors $d_k^{\phi_{k-1}}$ generated by the Broyden class methods are colinear to the vectors $d_k\PCG$ and that
  \begin{equation}\label{Higk}
    H_i g_k = H_0 g_k, ~i=0, ..., k-1.
  \end{equation}
  If the two vectors $d_k^{\phi_{k-1}}$ and $d_k\PCG$ are colinear, it follows that with an exact line search the two methods calculate the same iterates $s_k = \alpha_k\PCG d_k\PCG = \alpha_k^{\phi_k} d_k^{\phi_{k-1}}$.
  For $k=0$, $d_0^{\phi_k} = d_0\PCG = -H_0 g_0$ and hence, $s_0\PCG = s_0^{\phi_k}$.
  Assume that the induction hypothesis holds at iteration $k \geq 0$ and $g_{k+1} \neq 0$.
  Because \(\phi_k \neq \phi_k^c\), \Cref{proposition-2} implies that \(H_1, \ldots, H_k\) are nonsingular.
  Under the induction hypothesis, iterates $s_1, s_2, \ldots, s_k$ of all Broyden methods are the same of those generated by CG.
  Henceforth, $s_i^T g_{k+1} = 0$ by~\eqref{direcorhto} for $i=0,1,...,k$ and by~\eqref{vk} and~\eqref{Higk},
  \begin{equation}
    v_i^T g_{k+1} = \left(\frac{s_i^T}{s_i^T y_i} - \frac{y_i^T H_i}{y_i^T H_i y_i}\right) g_{k+1} = -\frac{y_i^T H_0 g_{k+1}}{y_i^T H_i y_i}.
  \end{equation}
  Then by~\eqref{broyden-family} and~\eqref{Higk},
  \begin{align}\label{Higk1}
    H_{i+1} g_{k+1} &= H_i g_{k+1} - \frac{H_i y_i y_i^T H_i g_{k+1}}{y_i^T H_i y_i} + \phi_i (y_i^T H_i y_i) v_i v_i^T g_{k+1}\\
    &= H_0 g_{k+1} - \frac{H_i y_i y_i^T H_0 g_{k+1}}{y_i^T H_i y_i} - \phi_i v_i y_i^T H_0 g_{k+1}.\notag
  \end{align}
  Furthermore, by the definition of \(y_k\) and~\eqref{gradortho}, $y_i^T H_0 g_{k+1} = 0$ for $i=0,1,...,k-1$. Then $H_i g_{k+1} = H_0 g_{k+1}$ for $i=0,1,...,k$, that proves~\eqref{Higk}.\\
  Still according to~\eqref{gradortho},
  \begin{equation}\label{ykgk}
    y_k^T H_0 g_{k+1} = g_{k+1}^T H_0 g_{k+1}.
  \end{equation}
  By induction hypothesis, there exists a scalar $\gamma_k^{\phi_{k-1}} \neq 0$ such that
  \begin{equation}\label{dkphi}
    d_k^{\phi_{k-1}} = \gamma_k^{\phi_{k-1}} d_k\PCG.
  \end{equation}
  According to~\eqref{direcorhto} and~\eqref{CG}, we have
  \begin{align}
    y_k^T d_k\PCG &= \left(g_{k+1}^T - g_k^T\right) \left(-H_0 g_k + \beta_k d_{k-1}\PCG\right) = g_k^T H_0 g_k,\label{ykdk}\\
    H_k y_k &= H_k \left(g_{k+1} - g_k\right) = H_0 g_{k+1} + d_k^{\phi_{k-1}} = H_0 g_{k+1} + \gamma_k^{\phi_{k-1}} d_k\PCG.
  \end{align}
  So,
  \begin{equation}\label{ykHkyk}
    y_k^T H_k y_k = g_{k+1}^T H_0 g_{k+1} + \gamma_k^{\phi_{k-1}} g_k^T H_0 g_k.
  \end{equation}
  By taking again~\eqref{Higk1} with $i = k$ and by~\eqref{vk} and~\eqref{ykgk},
  \begin{align}
    H_{k+1} g_{k+1} &= H_0 g_{k+1} - \frac{H_k y_k y_k^T H_0 g_{k+1}}{y_k^T H_k y_k} - \phi_k g_{k+1}^T H_0 g_{k+1} \left(\frac{s_k}{s_k^T y_k} - \frac{H_k y_k}{y_k^T H_k y_k}\right)\notag\\
    &= H_0 g_{k+1} - \frac{g_{k+1}^T H_0 g_{k+1} H_k y_k}{y_k^T H_k y_k} + \phi_k g_{k+1}^T H_0 g_{k+1}\left(\frac{H_k y_k}{y_k^T H_k y_k} - \frac{d_k\PCG}{y_k^T d_k\PCG}\right)\notag\\
    &= H_0 g_{k+1} \left(1 - \frac{g_{k+1}^T H_0 g_{k+1}}{y_k^T H_k y_k} + \frac{\phi_k g_{k+1}^T H_0 g_{k+1}}{y_k^T H_k y_k} \right)\\
    & ~~~ - d_k\PCG \left(\frac{ \gamma_k^{\phi_{k-1}} g_{k+1}^T H_0 g_{k+1}}{y_k^T H_k y_k} - \frac{\gamma_k^{\phi_{k-1}} \phi_k g_{k+1}^T H_0 g_{k+1}}{y_k^T H_k y_k} + \frac{\phi_k g_{k+1}^T H_0 g_{k+1}}{y_k^T d_k\PCG} \right)\notag\\
    &= \left(\frac{\gamma_k^{\phi_{k-1}} g_k^T H_0 g_k + \phi_k g_{k+1}^T H_0 g_{k+1} }{\gamma_k^{\phi_{k-1}} g_k^T H_0 g_k + g_{k+1}^T H_0 g_{k+1} } \right) \left(H_0 g_{k+1} - \frac{g_{k+1}^T H_0 g_{k+1}}{g_k^T H_0 g_k} d_k\PCG \right).\notag
  \end{align}
  So we have by~\eqref{direction} and~\eqref{CG} the relationship~\eqref{EQUATION}.
  Thanks to the assumptions~\eqref{cond1} and~\eqref{cond2}, $d_{k+1}\PCG$ and $d_{k+1}^{\phi_k}$ are well-defined.
  Finally, $\phi_k \neq \phi_ k^c$ brings by proposition \ref{proposition-1} that $H_{k+1} g_{k+1} \neq 0$ and the two directions are colinear.
  So, with an exact line search, we have $s_{k+1} = \alpha_{k+1}^{\phi_k} d_{k+1}^{\phi_k} = \alpha_{k+1}\PCG d_{k+1}\PCG$ under the assumption that $s_{k+1}^T A s_{k+1} \neq 0$.
\end{proof}

The following result generalizes \citep[Corollary~\(1\)]{broyden-1970} to \(\phi_k \neq \phi_k^c\).

\begin{corollary}%
  \label{cor:broyden-quadratic-termination}
  Let the assumptions of \Cref{theorem-1} be satisfied.
  Methods of the Broyden class inherit the conjugacy property as well as quadratic termination if \(j = n\).
  The secant equation
  \begin{equation}\label{vi}
    H_{k+1} y_i = s_i, \quad (i \leq j)
  \end{equation}
  holds for all previous iterates.
  It follows that if $n$ steps are performed and \(A\) is nonsingular, $H_n = A^{-1}$.
\end{corollary}

Clearly, assumptions on \(A\) in \Cref{theorem-1} and \Cref{cor:broyden-quadratic-termination} hold if \(A\) is SPD.

\begin{corollary}\label{theorem-2}
  Let \Cref{asm:q-strictly-convex,asm:broyden} be satisfied.
  A sufficient condition for~\eqref{cond2} to hold is that $\phi_k > \phi_k^c$.
  Thus,~\eqref{EQUATION} holds and quadratic termination occurs.
\end{corollary}

\begin{proof}
  If~\eqref{quadratic} is strictly convex, $s_k^T A s_k > 0$ and~\eqref{cond1} holds for all $k$.
  As a result of \Cref{proposition-2} and because $H_0 \succ 0$, $H_k \succ 0$ for all $k$. Thus, $y_k^T H_k y_k > 0$ and~\eqref{cond2} holds for all $k$.
  The result follows from \Cref{theorem-1}.
\end{proof}

\begin{lemma}\label{lemma-9}
  With the notation of \Cref{theorem-1} we can write
  \begin{equation}\label{phi-k-c}
    \phi_k^c = -\gamma_k^{\phi_{k-1}} \frac{g_k^T H_0 g_k}{g_{k+1}^T H_0 g_{k+1}}.
  \end{equation}
\end{lemma}

\begin{proof}
  By~\eqref{phikc},
  \begin{equation*}
    \phi_k^c = \frac{(y_k^T s_k)^2}{(y_k^T s_k)^2 - (y_k^T H_k y_k)(s_k^T B_k s_k)} = \frac{(y_k^T d_k^{\phi_{k-1}})^2}{(y_k^T d_k^{\phi_{k-1}})^2 - (y_k^T H_k y_k)({d_k^{\phi_{k-1}}}^T B_k d_k^{\phi_{k-1}})}.
  \end{equation*}
  Furthermore, $B_k d_k^{\phi_{k-1}} = -g_k$. So, by~\eqref{dkphi},
  \begin{equation}
    {d_k^{\phi_{k-1}}}^T B_k d_k^{\phi_{k-1}} = - \gamma_k^{\phi_{k-1}} g_k^T d_k\PCG = \gamma_k^{\phi_{k-1}} g_k^T H_0 g_k.
  \end{equation}
  It follows from~\eqref{ykdk} and~\eqref{ykHkyk} that
  \begin{align*}
    \phi_k^c & = \frac{\left(\gamma_k^{\phi_{k-1}} g_k^T H_0 g_k\right)^2}{\left(\gamma_k^{\phi_{k-1}} g_k^T H_0 g_k\right)^2 - \left(\gamma_k^{\phi_{k-1}} g_k^T H_0 g_k + g_{k+1}^T H_0 g_{k+1} \right) \gamma_k^{\phi_{k-1}} g_k^T H_0 g_k}
    \\ & = -\frac{\gamma_k^{\phi_{k-1}} g_k^T H_0 g_k}{g_{k+1}^T H_0 g_{k+1}}.
  \end{align*}
\end{proof}

\begin{corollary}\label{corollary-9}
  Let \Cref{asm:q-strictly-convex,asm:broyden} be satisfied.
  Then, $\gamma_k^{\phi_{k-1}} > 0$ provided $\phi_k > \phi_k^c$.
\end{corollary}

\begin{proof}
  By~\eqref{ykHkyk} and \Cref{proposition-2}, the denominator of $\gamma_k^{\phi_{k-1}}$ is positive because~\eqref{quadratic} is strictly convex, $H_0$ PD and $\phi_k > \phi_k^c$.
  If $\gamma_{k-1}^{\phi_k} > 0$,~\eqref{phi-k-c} implies $\gamma_k^{\phi_{k-1}} > 0$ if $\phi_k > \phi_k^c$. Finally, $\gamma_0^{\phi_k} = 1$.
  By induction, $\gamma_k^{\phi_{k-1}} > 0$ for all $k \geq 0$.
\end{proof}

\begin{corollary}%
  \label{corollary-10}
  Let \Cref{asm:q-strictly-convex,asm:broyden} be satisfied.
  Methods of the Broyden class with $\phi_k \geq 0$ are always well-defined.
\end{corollary}

\begin{proof}
  On a strictly convex quadratic function, it was proved by \Cref{corollary-9} that if $\phi_k > \phi_k^c$, then $\gamma_k^{\phi_{k-1}} > 0$ at each iteration.
  Thus by~\eqref{phi-k-c}, $\phi_k^c < 0$.
  It follows by \Cref{theorem-2} that the Broyden class methods are well-defined when $\phi_k \geq 0$.
\end{proof}

From \Cref{theorem-1}, if two methods of the Broyden class are applied to minimize~\eqref{quadratic}, one with parameter \(\phi_k^{(1)}\) and the other with parameter \(\phi_k^{(2)}\), the search directions at iteration \(k + 1\) are related via
\begin{equation}%
  \label{eq:dphi1-dphi2}
  d_{k+1}^{\phi_k^{(1)}} = \frac{ \gamma_{k+1}^{\phi_k^{(1)}} }{ \gamma_{k+1}^{\phi_k^{(2)}} } \, d_{k+1}^{\phi_k^{(2)}}.
\end{equation}

The next two results examine special cases.

\begin{proposition}
  Let \Cref{asm:q-strictly-convex,asm:broyden} be satisfied.
  Let $d_k\BFGS$ and $d_k\DFP$ be the directions generated respectively by the BFGS and DFP methods. Let $\gamma_k\BFGS$ and $\gamma_k\DFP$ the coefficients such that $d_k\BFGS = \gamma_k\BFGS d_k\PCG$ and $d_k\DFP = \gamma_k\DFP d_k\PCG$.
  Then,
  \begin{equation}\label{dkBFGS}
    d_{k+1}\BFGS = d_{k+1}\PCG,
  \end{equation}
  i.e., $\gamma_k\BFGS = 1$ for all $k$.
  In addition,
  \begin{equation}\label{DFPdk}
    d_{k+1}\DFP =  \frac{\gamma_k\DFP g_k^T H_0 g_k}{\gamma_k\DFP g_k^T H_0 g_k + g_{k+1}^T H_0 g_{k+1}} d_{k+1}\PCG.
  \end{equation}
\end{proposition}

\begin{proof}
  The BFGS and DFP methods correspond to $\phi_k = 1$ and $\phi_k = 0$ in~\eqref{EQUATION}.
\end{proof}

\begin{proposition}
  Let \Cref{asm:q-strictly-convex,asm:broyden} be satisfied.
  Let $d_k\SRI$ be the direction generated by the SR1 method and let $\gamma_k\SRI$ be the coefficient such that $d_k\SRI = \gamma_k\SRI d_k\PCG$.
  If
  \begin{equation}
    \left(s_k - H_k y_k \right)^T y_k \neq 0,
  \end{equation}
  then $\phi_k\SRI$ is well-defined and~\eqref{phi-SR1} exists.
  In addition,
  \begin{equation}
    d_{k+1}\SRI = \frac{\left(\gamma_k\SRI - \alpha_k\PCG \right) g_k^T H_0 g_k}{\left(\gamma_k\SRI - \alpha_k\PCG \right) g_k^T H_0  g_k + g_{k+1}^T H_0 g_{k+1}} d_{k+1}\PCG.\label{SR1dk}
  \end{equation}
\end{proposition}
\begin{proof}
  The assumption $(s_k - H_k y_k)^T y_k \neq 0$ is precisely the condition under which the SR1 parameter~\eqref{phi-SR1} is well-defined, and hence the SR1 update belongs to the Broyden family with $\phi_k = \phi_k\SRI$.
  Since the iterates generated by the Broyden family coincide with those of PCG by \Cref{theorem-1}, we have \(s_k = \alpha_k\PCG d_k\PCG\).
  Using~\eqref{ykdk} and~\eqref{ykHkyk}, with $\gamma_k^{\phi_{k-1}} = \gamma_k\SRI$, gives
  \begin{align*}
    y_k^T s_k &= \alpha_k\PCG y_k^T d_k\PCG
    = \alpha_k\PCG g_k^T H_0 g_k,\\
    y_k^T H_k y_k &= \gamma_k\SRI g_k^T H_0 g_k + g_{k+1}^T H_0 g_{k+1}.
  \end{align*}
  Therefore,
  \begin{align*}
    \phi_k\SRI
    &= \frac{\alpha_k\PCG g_k^T H_0 g_k}
    {\left(\alpha_k\PCG - \gamma_k\SRI\right) g_k^T H_0 g_k - g_{k+1}^T H_0 g_{k+1}}\\
    &= - \frac{\alpha_k\PCG g_k^T H_0 g_k}
    {\left(\gamma_k\SRI - \alpha_k\PCG\right) g_k^T H_0 g_k + g_{k+1}^T H_0 g_{k+1}}.
  \end{align*}
  Substituting this expression for $\phi_k\SRI$ in the  numerator of~\eqref{EQUATION} yields
  \begin{multline*}
    \gamma_k\SRI g_k^T H_0 g_k + \phi_k\SRI g_{k+1}^T H_0 g_{k+1}
    \\ = \gamma_k\SRI g_k^T H_0 g_k
    - \frac{\alpha_k\PCG g_k^T H_0 g_k\, g_{k+1}^T H_0 g_{k+1}}
    {\left(\gamma_k\SRI - \alpha_k\PCG\right) g_k^T H_0 g_k + g_{k+1}^T H_0 g_{k+1}}
    \\ = \frac{\gamma_k\SRI g_k^T H_0 g_k
    \left[\left(\gamma_k\SRI - \alpha_k\PCG\right) g_k^T H_0 g_k + g_{k+1}^T H_0 g_{k+1}\right]
    - \alpha_k\PCG g_k^T H_0 g_k\, g_{k+1}^T H_0 g_{k+1}}
    {\left(\gamma_k\SRI - \alpha_k\PCG\right) g_k^T H_0 g_k + g_{k+1}^T H_0 g_{k+1}}
    \\ = \frac{\gamma_k\SRI\left(\gamma_k\SRI - \alpha_k\PCG\right) \left(g_k^T H_0 g_k\right)^2
    + \left(\gamma_k\SRI - \alpha_k\PCG\right) g_k^T H_0 g_k\, g_{k+1}^T H_0 g_{k+1}}
    {\left(\gamma_k\SRI - \alpha_k\PCG\right) g_k^T H_0 g_k + g_{k+1}^T H_0 g_{k+1}}
    \\ = \frac{\left(\gamma_k\SRI - \alpha_k\PCG\right) g_k^T H_0 g_k
    \left(\gamma_k\SRI g_k^T H_0 g_k + g_{k+1}^T H_0 g_{k+1}\right)}
    {\left(\gamma_k\SRI - \alpha_k\PCG\right) g_k^T H_0 g_k + g_{k+1}^T H_0 g_{k+1}}.
  \end{multline*}
The result follows by dividing the above by $\gamma_k\SRI g_k^T H_0 g_k + g_{k+1}^T H_0 g_{k+1}$.
\end{proof}

The next result concerns methods of the \emph{convex} Broyden class, i.e., those for which \(0 \leq \phi_k \leq 1\).

\begin{theorem}%
  \label{prop:optimal-step-lenghts}
  Let \Cref{asm:broyden} be satisfied.
  Consider two methods of the Broyden class, one with parameter $\phi_k^{(1)}$ and the other with parameter $\phi_k^{(2)}$ such that $0 \leq \phi_k^{(1)} \leq \phi_k^{(2)} \leq 1$ for all $k$.
  Let $\gamma_k^{(1)}$and $\gamma_k^{(2)}$ be the coefficients defined by~\eqref{EQUATION} for each method.
  Then
  \begin{equation}\label{Gamma}
    0 < \gamma_k^{(1)} \leq \gamma_k^{(2)} \leq 1.
  \end{equation}
  Let $\alpha_k^{(1)}$ and $\alpha_k^{(2)}$ be the optimal step lengths associated with each method at iteration \(k\). We have
  \begin{subequations}\label{61}
    \begin{align}
      \alpha_k^{(2)} \leq \alpha_k^{(1)} ~~&\text{if } s_k^T A s_k > 0,\\
      \alpha_k^{(1)} \leq \alpha_k^{(2)} ~~&\text{if } s_k^T A s_k < 0.
    \end{align}
  \end{subequations}
  Finally, let $\alpha_k\PCG$ the optimal step length associated to PCG with preconditioner $H_0$ and let $\alpha_k^{\phi_k}$ be the optimal step length of any given  method in the convex Broyden class. We have
  \begin{subequations}\label{62}
    \begin{align}
      0 < \alpha_k\PCG \leq \alpha_k^{\phi_k} ~~&\text{if } s_k^T A s_k > 0,\\
      \alpha_k^{\phi_k} \leq \alpha_k\PCG < 0 ~~&\text{if } s_k^T A s_k < 0.
    \end{align}
  \end{subequations}
\end{theorem}

\begin{proof}
  By \Cref{theorem-2} and \Cref{lemma-9}, $\gamma_{k+1}^{\phi_k} > 0$ if $\phi_k > \phi_k^c$ and $\gamma_k^{\phi_{k-1}} > 0$.
  Because $\phi_k^c < 0$ when $\gamma_k^{\phi_{k-1}} \geq 0$, it is equivalent to $\phi_k \geq 0$ and $\gamma_k^{\phi_{k-1}} > 0$.
  In addition $\gamma_0^{\phi_k} = 1 > 0$ so $\gamma_k^{\phi_{k-1}} > 0$ for all $k$.
  By noting that
  \begin{equation}\label{alpha_CG-over-alpha_phi}
    \gamma_k^{\phi_{k-1}} = \alpha_k\PCG / \alpha_k^{\phi_k},
  \end{equation}
  it follows that $\alpha_k^{\phi_k}$ has the same sign as $\alpha_k\PCG$, and by transitivity $\alpha_k^{(1)}$ and $\alpha_k^{(2)}$ have the same sign as well.
  Let $\gamma_{k+1}^{(1)}$ and $\gamma_{k+1}^{(2)}$ obtained by replacing $\phi_k$ by $\phi_k^{(1)}$ and $\phi_k^{(2)}$ in~\eqref{EQUATION}.
  Showing by induction that for all $k$, $\gamma_k^{(1)} \leq \gamma_k^{(2)}$ if $0 \leq \phi_k^{(1)} \leq \phi_k^{(2)} \leq 1$.
  To begin, we have $\gamma_0^{(1)} = \gamma_0^{(2)} = 1$ as $d_0 = -H_0 g_0$.
  Assuming that $\gamma_k^{(1)} \leq \gamma_k^{(2)}$ until one $k \geq 0$ and showing that it holds true at the rank $k+1$.
  We have for $i=1,2$
  \begin{align}
    \gamma_{k+1}^{(i)} &= \frac{\gamma_k^{(i)} g_k^T H_0 g_k + \phi_k^{(i)} g_{k+1}^T H_0  g_{k+1}}{\gamma_k^{(i)} g_k^ T H_0 g_k + g_{k+1}^T H_0 g_{k+1}} = 1 + \frac{g_{k+1}^T H_0 g_{k+1}\left(\phi_k^{(i)}-1\right)}{\gamma_k^{(i)} g_k^T H_0 g_k + g_{k+1}^T H_0 g_{k+1}}\notag\\
    &= 1 - \frac{g_{k+1}^T H_0 g_{k+1}\left|\phi_k^{(i)} - 1\right|}{\gamma_k^{(i)} g_k^T H_0 g_k + g_{k+1}^T H_0 g_{k+1}}.\notag
  \end{align}
  The last relation is due to the fact that $\phi_k^{(1)} \in [0,1]$.
  Because $0 \leq \phi_k^{(1)} \leq \phi_k^{(2)} \leq 1$,
  \begin{equation}
    \left| \phi_k^{(1)} -1 \right| \geq \left| \phi_k^{(2)} -1 \right|.
  \end{equation}
  As by induction hypothesis $0 < \gamma_k^{(1)} \leq \gamma_k^{(2)}$, then
  \begin{equation}
    \gamma_k^{(1)} g_k^T H_0 g_k + g_{k+1}^T H_0 g_{k+1} \leq \gamma_k^{(2)} g_k^T H_0 g_k + g_{k+1}^T H_0 g_{k+1}.
  \end{equation}
  It follows that
  \begin{equation}
    \gamma_{k+1}^{(1)} \leq 1 - \frac{g_{k+1}^T H_0 g_{k+1}\left|\phi_k^{(2)}-1\right|}{\gamma_k^{(2)} g_k^T H_0 g_k + g_{k+1}^T H_0 g_{k+1}} = \gamma_{k+1}^{(2)}.
  \end{equation}
  In addition, it is clear that $\gamma_k^{\phi_{k-1}} \leq 1$ for all $0 \leq \phi_k \leq 1$.
  By induction, we have shown~\eqref{Gamma}.
  Furthermore, by~\eqref{alpha_CG-over-alpha_phi} we have
  \begin{equation}
    \left|\alpha_k^{(2)}\right| \leq \left|\alpha_k^{(1)}\right|.
  \end{equation}
  We obtain~\eqref{61} by noting that $\alpha_k^{(1)}$ and $\alpha_k^{(2)}$ have the same sign as $\alpha_k\PCG$, which in turn has the same sign as $s_k^T A s_k$ by~\eqref{exact-linesearch} and~\eqref{CG}.
  To finish, we establish~\eqref{62} by noting with~\eqref{dkBFGS} that $\alpha_k\PCG = \alpha_k\BFGS$ and by replacing $\alpha_k^{(2)}$ by $\alpha_k\PCG$ with $\phi_k^{(2)} = 1$.
\end{proof}

The step length behavior of \Cref{prop:optimal-step-lenghts} is illustrated in \Cref{alphaFig}.

\begin{figure}[ht]
  \begin{center}
    \includegraphics[scale = 0.60]{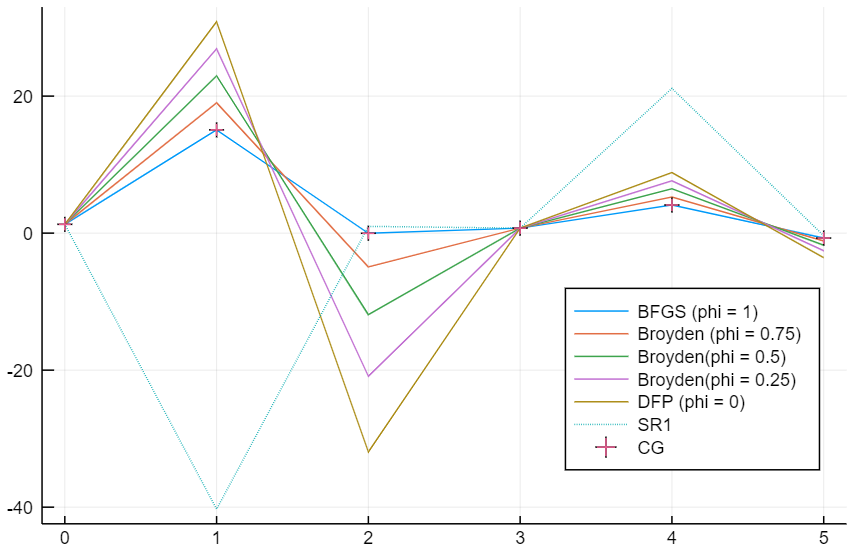}
    \caption{\label{alphaFig} Optimal step lengths $\alpha_k$ as a function of the iteration number $k$. The optimal step lengths of the convex Broyden class methods are ordered whereas the optimal step lengths of SR1 show an erratic behaviour.}
  \end{center}
\end{figure}

\section{The Limited-Memory BFGS Method}

Results of the previous section suggest that methods of the Broyden class could be used in place of PCG wherever it is typically used.
One such usage is in the computation of steps for the minimization of a nonlinear, potentially nonconvex function.
However, the usual disadvantage of Broyden methods is that they tend to generate dense approximations \(H_k\), and are thus impractical for large-scale problems.
In this section, we examine the relationship between the most widely-used limited-memory Broyden method and PCG.

The LBFGS($m$) method is a limited-memory quasi-Newton method.
Instead of storing a matrix $H_k$ of size $n \times n$, we store a limited amount $m$ of couples $(s_i, y_i)$ to reconstruct an approximation of the product $H_k g_k$.

Let
\begin{equation}
  U_k = I - \rho_k y_k s_k^T, \quad \text{and} \quad \rho_k = 1/y_k^T s_k.
\end{equation}
In~\citep{Liu-Nocedal-1989}, it is shown for $m = k$ and $H_k^0 = H_0$, that the matrix
\begin{align}\label{LBFGS}
  H_k^m = &\left(U_{k-1}^T...U_{k-m}^T\right)H_k^0\left(U_{k-m}...U_{k-1}\right)\\
  &+ \rho_{k-m} \left(U_{k-1}^T...U_{k-m+1}^T\right)s_{k-m} s_{k-m}^T \left(U_{k-m+1}...U_{k-1}\right)\notag\\
  &+ \rho_{k-m+1} \left(U_{k-1}^T...U_{k-m+2}^T\right)s_{k-m+1} s_{k-m+1}^T \left(U_{k-m+2}...U_{k-1}\right)\notag\\
  &+ \cdots + \rho_{k-1} s_{k-1} s_{k-1}^T\notag
\end{align}
is exactly the one calculated by the BFGS method~\eqref{BFGS}.
If we store a small number of the most recent couples $(s_i, y_i)$, we lose some information necessary to reconstruct $H_k\BFGS$ but we hope to have kept relevant information to determine a good descent direction $d_k^m = - H_k^m g_k$.


As it turns out, the LBFGS($m$) method possesses the same good properties as BFGS in relation with PCG.

\begin{theorem}\label{theorem-12}
  Assume that \(A\) is not necessarily PD.
  Let $x_0$ be any starting point and let $H_0$ be any SPD matrix.
  Consider, on the one hand, PCG applied to~\eqref{quadratic} from \(x_0\) with preconditioner \(H_0\).
  Consider, on the other hand, the LBFGS($m$) method with \(m \geq 1\) and exact line search to minimize~\eqref{quadratic} with \(H_k^0 = H_0\) for all \(k\).
  Assume that iterations \(k = 0, \ldots, j\) of both methods are well defined.
  If $s_k^T y_k \neq 0$ for all $k = 0, \ldots, j$, then, as long as $g_k \neq 0$, the search directions \(d_k^m\), step lengths \(\alpha_k^m\), and iterates $s_k$ generated by LBFGS($m$) are identical to those generated by PCG.
  Consequently, if \(j = n\), quadratic termination occurs.
\end{theorem}

\begin{proof}
  As in \Cref{theorem-1}, we show that the directions $d_k^m$ generated by LBFGS($m$) are colinear to those generated by PCG~\eqref{CG}.
  We note that for $k=0$, $d_0^m = d_0\PCG = -H_0 g_0$.
  Assuming that we have the colinearity until a rank $k \geq 0$ and showing that it holds at the rank $k+1$ if $g_{k+1} \neq 0$.
  Under the induction hypothesis, the current points $x_1,..,x_{k+1}$ generated by LBFGS($m$) are identical to those generated by CG.
  Henceforth, by~\eqref{direcorhto}, $s_i^T g_{k+1} = 0$ for all $i=0,...,k$.
  So, we have
  \begin{equation}
    U_i g_{k+1} = \left(I - \rho_i y_i s_i^T \right) g_{k+1} = g_{k+1}, ~i=0,...,k.
  \end{equation}
  By~\eqref{gradortho}
  \begin{equation}
    U_i^T H_0 g_{k+1} = \left(I - \rho_i s_i y_i^T \right) H_0 g_{k+1} = H_0 g_{k+1}, ~i=0,...,k-1.
  \end{equation}
  By applying these two relations to~\eqref{LBFGS}, we have,
  \begin{equation}\label{HkgkLBFGS}
    H_{k+1}^m g_{k+1} = \left(U_k^T...U_{k+1-m}^T\right) H_{k+1}^0 g_{k+1}.
  \end{equation}
  With $H_{k+1}^0 = H_0$, it follows that
  \begin{align}
    -H_{k+1}^m g_{k+1} &= -U_k^T H_0 g_{k+1} = -H_0 g_{k+1} + \rho_k s_k \left(g_{k+1}^T - g_k^T\right) H_0 g_{k+1}\notag\\
    &= -H_0 g_{k+1} + \frac{g_{k+1}^T H_0 g_{k+1}}{y_k^T s_k} s_k\notag = -H_0 g_{k+1} + \frac{g_{k+1}^T H_0 g_{k+1}}{y_k^T d_k\PCG} d_k\PCG. \notag
  \end{align}
  So by~\eqref{ykdk} and~\eqref{CG},
  \begin{equation}
    d_{k+1}^m = -H_0 g_{k+1} + \frac{g_{k+1}^T H_0 g_{k+1}}{g_k^T H_0 g_k} d_k\PCG = d_{k+1}\PCG.
  \end{equation}
  So $d_{k+1}^m$ and $d_{k+1}\PCG$ are colinear because $H_0$ is non singular.
  It follows with an exact line search that the iterates $s_k$ are identical for both methods and for any $m \geq 1$.
\end{proof}

Clearly, a sufficient condition for both methods to be well defined in \Cref{theorem-12} is that \(A\) be PD.

\Cref{theorem-12} can be interpreted as follow.
In the BFGS update~\eqref{dkBFGS}, only the last couple $(y_k, s_k)$ is useful to calculate the descent direction $d_k^m$ in the quadratic case.
However, we will see in numerical experiments that the other couples in storage are important for the robustness and accuracy of LBFGS($m$).

\section{The Limited-Memory Full Orthogonalization Method}

In this section, we turn to another class of methods that are known, or can be shown, to coincide with PCG when applied to~\eqref{quadratic}: Krylov-subspace methods based on the Arnoldi process. We first present the general framework and then specialize to the case where $A$ is SPD.

\subsection{Arnoldi Process}
\phantomsection  

Given an initial vector $v_1$, the $k$th Krylov subspace is
\(\mathcal{K}_k(A, v_1) := \mathrm{span}\{v_1, Av_1, \ldots, A^{k-1} v_1\}\), \(k \geq 1\).
The Arnoldi process builds an orthonormal basis of $\mathcal{K}_k(A, v_1)$ via \Cref{Algo:Arnoldi}.

\begin{algorithm}[ht]
  \caption{%
    \label{Algo:Arnoldi}
    Arnoldi Process.
  }
  \begin{algorithmic}[1]
    \Require{} $v_1$ such that $\|v_1\|_2 = 1$
    \For{\(k = 1, 2, \dots\)}
    \State \(w = A v_k\)
    \For{\(i = 1, \dots, k\)}
    \State set \(t_{i, k} = v_i^T w\) and \(w \leftarrow w - t_{i,k} v_i\)
    \EndFor
    \State \(t_{k+1,k} v_{k+1} = w\)
    \Comment{\(t_{k+1,k} > 0\) such that \(\|v_{k+1}\| = 1\)}
    \EndFor
  \end{algorithmic}
\end{algorithm}

By construction, \(k\) iterations of \Cref{Algo:Arnoldi} in exact arithmetic generate $V_k = \begin{bmatrix} v_1 & \ldots & v_k \end{bmatrix} \in \R^{n\times k}$ with orthonormal columns and a $k \times k$ upper Hessenberg matrix \(T_k\) whose nonzero entries are the $t_{i,j}$ such that
\begin{subequations}
\begin{align}
  \label{eq:arnoldi_process}
  A V_k & = V_k T_k + t_{k+1, k} v_{k+1} e_k^T,
  \\
  \label{eq:hesseberg_matrix}
  V_k^T A V_k & = T_k.
\end{align}
\end{subequations}
In floating-point arithmetic,~\eqref{eq:arnoldi_process} holds to machine precision, but~\eqref{eq:hesseberg_matrix} does not hold.
We slightly depart from standard notation in the literature here and denote the Hessenberg matrix \(T_k\).
The reason for this notation is that, in the present context, we apply \Cref{Algo:Arnoldi} to symmetric \(A\), in which case \(T_k\) becomes tridiagonal in exact arithmetic---but not in floating-point arithmetic.

\subsection{The Full orthogonalization method (FOM)}

Arnoldi-based methods are appropriate to solve a linear system \(Ax = b\) where \(A\) is square, and not necessarily symmetric.
Given an initial guess \(x_0\), we let \(r_0 = b - A x_0\) be the initial residual.
It is common to define \(\beta = \|r_0\|_2\).

All such methods seek an approximate solution of the form
\begin{equation}\label{eq:x_l}
  x_k = x_0 + V_k w_k,
\end{equation}
where $V_k$ is generated by \Cref{Algo:Arnoldi} with $v_1 = r_0 / \beta$, and $w_k \in \R^k$ is chosen to enforce a suitable condition.

We can express the residual $r_k = b - A x_k$ in terms of $w_k$ as
\begin{equation}\label{eq:r_l_sym}
  r_k = b - A x_k
  = b - A (x_0 + V_k w_k)
  = r_0 - A V_k w_k
  = \beta v_1 - A V_k w_k.
\end{equation}

The full orthogonalization method (FOM) is one such method, and computes an approximate solution~\eqref{eq:x_l} by imposing the Galerkin condition on the residual, i.e., $r_k \perp \mathcal{K}_k(A,r_0)$ for \(k \geq 1\).
In exact arithmetic, $w_k$ is obtained by solving the projected system
\begin{equation}\label{eq:orthogonality_property}
  V_k^T r_k = V_k^T ( \beta v_0 - A V_k w_k) = 0
  \iff
  V_k^T A V_k w_k =  \beta e_1,
\end{equation}
where $e_1 = [1,0,\ldots,0]\in \R^{k}$.
\Cref{alg:FOM} summarizes the procedure.

\begin{algorithm}[ht]
  \caption{%
    \label{alg:FOM}
    Full Orthogonalization Method (FOM).
  }
  \begin{algorithmic}[1]
    \State Choose \(x_0 \in \R^n\), compute \(r_0 = b - A x_0\), and set \(k \geq 1\).
    \State Obtain \(V_k\) and \(T_k\) from \Cref{Algo:Arnoldi} initialized with \(v_1 = r_0 / \beta\).
    \State%
    \label{stp:fom-system}%
    Solve $T_k w_k = \beta e_1$ and set $x_k\FOM = x_0 + V_k w_k$
  \end{algorithmic}
\end{algorithm}

Our first result is well known in linear algebra circles, but perhaps not so well known in the optimization community.

\begin{proposition}%
  \label{prop:fom=cg}
  Let \(A = A^T\), \(b \in \R^n\) and \(x_0 \in \R^n\).
  Suppose \Cref{alg:FOM} is well defined for values \(k = 1, \ldots, j\) in the sense that each \(T_k\) is nonsingular.
  Then, each \(x_k\FOM = x_k\PCG\) with \(H_0 = I\) and the same initial \(x_0\).
\end{proposition}

\begin{proof}
  When \(A = A^T\),~\eqref{eq:hesseberg_matrix} shows that \(T_k = T_k^T\), and hence, is tridiagonal, at least in exact arithmetic.
  Indeed, in that case, \Cref{Algo:Arnoldi} coincides with the \citet{lanczos-1950} process for symmetric matrices.
  Moreover, Line~\ref{stp:fom-system} of \Cref{alg:FOM} yields precisely \(x_k\PCG\) with \(H_0 = I\) \citep{paige-saunders-1975}.
\end{proof}

A sufficient condition for \Cref{prop:fom=cg} to hold is that \(A\) be PD.
Indeed, in that case,~\eqref{eq:hesseberg_matrix} shows that \(T_k\) is also PD, and hence nonsingular.
In the remainder of this section and the next, any reference to \Cref{Algo:PCG} implicitly assumes that \(H_0 = I\) and that both methods are initialized from the same \(x_0\).

\begin{proposition}\label{eq:arnoldi_res}
  The residual in \Cref{alg:FOM} satisfies
  \begin{equation*}
    r_0 = \beta v_1, \qquad
    r_k = - t_{k+1, k} (e_k^T w_k) v_{k+1} \quad (k \geq 1).
  \end{equation*}
  where $e_k$ denotes the $k$th canonical basis vector of $\R^k$.
\end{proposition}

\begin{proof}
  From~\eqref{eq:r_l_sym},~\eqref{eq:arnoldi_process} and $T_k w_k = \beta e_1$,
  \begin{equation*}
    r_k
    = \beta v_1 - \bigl(V_k T_k + t_{k+1, k} v_{k+1} e_k^T\bigr)w_k
    = \beta v_1 - V_k \beta e_1 - t_{k+1, k} v_{k+1} e_k^T w_k.
  \end{equation*}
  Since $V_k e_1 = v_1$, the first two terms cancel, which establishes the result.
\end{proof}

Consider the LU factorization $T_k = L_k U_k$, where $L_k$, with entries $\ell_{i,j}$, is unit lower bidiagonal, and $U_k$, with entries $u_{i,j}$, is upper triangular.
To solve \(T_k w_k = \beta e_1\), we first solve \(L_k z_k = \beta e_1\) followed by \(U_k w_k = z_k\).
It will be useful to express $r_k$ in terms of \(L_k\), \(U_k\), and \(z_k\).

\begin{corollary}\label{prop:diom_res}
  For \(k \geq 1\),
  \begin{equation*}
    r_k = -t_{k+1, k}
    \frac{
      \zeta_k
    }{
      u_{k,k}
    }
    v_{k+1}
    = -t_{k+1, k}
    \frac{
      \beta
    }{
      u_{k,k}
    }
    \Big(\prod_{i=2}^{k}(-\ell_{i,i-1})\Big) v_{k+1},
  \end{equation*}
  where \(\zeta_k\) is the last component of \(z_k\).
\end{corollary}

\begin{proof}
  Forward substitution in $L_k z_k = \beta e_1$ yields $\zeta_k = \beta \prod_{i=2}^{k}(-\ell_{i,i-1})$.
  By \Cref{eq:arnoldi_res}, we must compute $\omega_k := e_k^T w_k$, i.e., the \(k\)-th component of \(w_k\).
  Back substitution into $U_k w_k = z_k$ gives \(\omega_k = \zeta_k / u_{k,k}\).
\end{proof}

\Cref{alg:FOM} is akin to quasi-Newton methods in the sense that it typically generates a dense upper Hessenberg matrix \(T_k\), and storage requirements grow quadratically with the number of iterations.
Although we may expect that extra off-diagonals in the upper triangle of \(T_k\) that appear from loss of orthogonality between the basis vectors \(v_k\) in floating-point arithmetic would improve robustness compared to PCG, which simply ignores them, such memory requirements may be prohibitive in practice.
In the next section, we review a limited-memory variant of FOM, which is then akin to limited-memory quasi-Newton methods.

In order to establish a connection with PCG, we now introduce the matrix \(P_k\).
Under the assumption that \(T_k\) is nonsingular, the update of \(x_k\FOM\) can be written
\begin{equation}%
  \label{eq:xk-fom-update}
  x_k\FOM = x_0 + V_k w_k = x_0 + V_k U_k^{-1} L_k^{-1} (\beta e_1) = x_0 + P_k z_k = x_{k-1}\FOM + \zeta_k p_k,
\end{equation}
where
\begin{equation}%
  \label{eq:def-P-z}
  P_k := V_k U_k^{-1} = \begin{bmatrix} p_1\DIOM \cdots p_k\DIOM \end{bmatrix},
  \quad \text{and} \quad
  z_k := L_k^{-1} (\beta e_1) = (\zeta_1, \ldots, \zeta_k).
\end{equation}
Because \(U_k\) is upper triangular, the columns of \(P_k\) can be determined from the linear system with multiple right-hand sides
\begin{equation}%
  \label{eq:determine-P}
  U_k^T P_k^T = V_k^T
\end{equation}
as the iterations progress.

\subsection{The direct incomplete orthogonalization method (DIOM)}\label{DIOM}

Instead of orthogonalizing each new \(v_k\) against all previous ones, DIOM($m$) uses a sliding window of size $m$, reducing storage and computational costs while retaining the most recent information of the Krylov-subspace structure.
That is similar to how limited-memory quasi-Newton methods only retain the most recent curvature information.
Such strategy is quite different from \emph{restarted} FOM, which is sometimes called FOM(\(m\)), and which simply discards all the \(v_k\) vectors when it reaches iteration \(m\), only to restart the process from \(x_m\FOM\) as new initial guess.
By \Cref{prop:fom=cg}, only the first \(m\) iterations of FOM(\(m\)), if they are well defined, coincide with those of CG.
Afterwards, the two methods differ.
We do not investigate restarted methods further.

In DIOM(\(m\)), instead of being upper Hessenberg, \(T_k\) becomes a banded upper Hessenberg matrix with upper semi-bandwidth \(m\), and $x_k$ can be updated using an LU factorization that can be updated cheaply from one iteration to the next.

We summarize the resulting iteration as \Cref{Algo:DIOM}.

\begin{algorithm}[ht]
  \caption{DIOM(\(m\))}
  \label{Algo:DIOM}
  \begin{algorithmic}[1]
    \State Choose $x_0$, set $r_0 = b-A x_0$, $v_1 \gets r_0 / \beta$, \(\zeta_1 = \beta\), and \(m \geq 1\).
    \For{$k = 1, 2,\ldots$}
    \State \(m_k = \max(1, \, k - m + 1)\)
    \State%
    \label{stp:begin-arnoldi}%
    \(w = A v_k\) \Comment{incomplete orthogonalization process}
    \For{\(i = m_k, \dots, k\)}
    \State set \(t_{i,k} = v_i^T w\) and \(w \leftarrow w - t_{ik} v_i\)
    \EndFor
    \State%
    \label{stp:end-arnoldi}
    \(t_{k+1,k} v_{k+1} = w\)
    \Comment{\(t_{k+1,k} > 0\) such that \(\|v_{k+1}\| = 1\)}
    \For{$i = m_k,\ldots,k$}%
    \label{stp:begin-lu}
    \Comment{update LU factorization of \(T_k\)}
    \If{$i = 1$}
    $u_{1,k} = t_{1,k}$  \Comment{update last column of $U_k$}
    \Else \
    $u_{i,k} = t_{i,k} - \ell_{i,i-1} \, u_{i-1,k}$ \label{line:updateU}
    \EndIf%
    \label{stp:end-lu}
    \EndFor%
    \label{stp:ukk}%
    \If{$u_{k,k}=0$} \textbf{stop} \Comment{abort: \(A\) is singular}
    \EndIf
    \State%
    \label{stp:pk-diom}%
    $p_k\DIOM = \Big(v_k - \sum_{i = m_k}^{k-1} u_{i,k} p_i\DIOM\Big) / u_{k,k}$
    \Comment{last equation of~\eqref{eq:determine-P}}
    \State $x_k\DIOM = x_{k-1}\DIOM + \zeta_k p_k\DIOM$ \Comment{update iterate}
    \State%
    \label{stp:lkp1k}%
    $\ell_{k+1,k} = t_{k+1,k} / u_{k,k}$ \Comment{update the new subdiagonal of $L_k$}
    \State%
    \label{stp:zeta-k}%
    $\zeta_{k+1} = -\ell_{k+1,k} \, \zeta_k$ \Comment{last component of \(z_{k+1}\)}
    \EndFor
  \end{algorithmic}
\end{algorithm}

Similarly to \Cref{prop:fom=cg}, we have the following equivalence.

\begin{proposition}%
  \label{prop:diom=cg}
  Let \(A = A^T\), \(b \in \R^n\) and \(x_0 \in \R^n\).
  Suppose \Cref{Algo:DIOM} is well defined for values \(k = 1, \ldots, j\) in the sense that each \(T_k\) is nonsingular.
  Then, each \(x_k\DIOM = x_k\PCG\) with \(H_0 = I\) and the same initial \(x_0\) for any \(m \geq 1\).
\end{proposition}

\begin{proof}
  The proof is similar to that of \Cref{prop:fom=cg} by noting that \(T_k\) is banded upper Hessenberg with upper semi-bandwidth \(m\).
  Hence, in exact arithmetic, \(T_k = T_k^T\) is tridiagonal, and the equivalence with PCG follows the same arguments.
\end{proof}

Because DIOM(\(m\)) orthogonalizes the most recent basis vector \(v_k\) against the \(m\) previous vectors, it can be interpreted as a form of reorthogonalization for PCG, albeit one that follows naturally from \Cref{Algo:Arnoldi}.

If \(T_k\), \(k = 1, \ldots, j\) are nonsingular, each \(u_{k,k} \neq 0\), and \(U_k = D_k \hat{U}_k\), where \(D_k\) is diagonal and \(\hat{U}_k\) is unit upper triangular.
The elements on the diagonal of \(D_k\) are \(u_{1,1}, \ldots, u_{k,k}\).
By~\eqref{eq:hesseberg_matrix} and~\eqref{eq:def-P-z}, \(P_k^T A P_k = U_k^{-T} L_k = \hat{U}_k^{-T} D_k^{-1} L_k\).
If \(A = A^T\), it must be that \(\hat{U}_k = L_k^T\), and hence \(P_k^T A P_k = D_k^{-1}\) is diagonal, which shows that the directions \(p_1\DIOM, \ldots, p_k\DIOM\) are conjugate.
In particular, if each \(T_k\), \(k = 1, \ldots, j\) were PD, as would occur, among other possible situations, if \(A\) were itself PD, \(D_k\) would also be PD, i.e., each \(u_{k, k} > 0\).
In the next section, we state properties of \Cref{Algo:DIOM} that hold when such positive curvature is observed.

\subsection{\texorpdfstring{DIOM(\(m\))}{DIOM(m)} in the presence of positive curvature}\label{subsec:DIOM_TR}

Under the assumption that positive curvature is observed over iterations \(k = 1, \ldots, j\), we show that $p_k\DIOM$ is not necessarily a descent direction for~\eqref{quadratic}, and relate it to $d_k\PCG$.

\begin{lemma}%
  \label{lem:zetak}
  For $k=1, 2, \ldots, j$, let \(P_k^T A P_k\) be SPD.
  Then, $\zeta_k = (-1)^{k-1} |\zeta_k|$, \(k = 1, \ldots, j\).
\end{lemma}

\begin{proof}
  For $k = 1$, $\zeta_1 = \beta = (-1)^0 \, \beta$.
  Assume that the statement holds for $k-1$, i.e., $\zeta_{k-1} = (-1)^{k-2} |\zeta_{k-1}|$.
  Then,
  \[
    \zeta_k = -\ell_{k+1,k} \, \zeta_{k-1}
    = -\ell_{k+1,k} \, (-1)^{k-2} |\zeta_{k-1}|
    = (-1)^{k-1} \, \ell_{k+1,k} |\zeta_{k-1}|. 
  \] 
  Because \(P_k^T A P_k\) is PD, each \(u_{k,k} > 0\), and hence, each $\ell_{k+1,k}\ge 0$.
  Thus, $\zeta_k = (-1)^{k-1} |\zeta_k|$.
\end{proof}

In \Cref{Algo:DIOM}, each \(p_k\DIOM \in \mathcal{K}_k(A, v_1)\), and hence, by the Galerkin condition,
\begin{equation}%
  \label{eq:p-orthogonal}
  r_j^T p_k\DIOM = 0 \quad j \geq 1, \quad k = 1, \ldots, j.
\end{equation}

\begin{theorem}\label{rkpk}
  For $k=1, 2, \ldots, j$, let \(P_k^T A P_k\) be SPD.
  Then,
  \[
    g_0^T p_1\DIOM = -\beta / u_{1,1},
    \mathhfill \text{and} \mathhfill
    g_k^T p_{k+1}\DIOM = 
    g_k^T p_{k+1}\DIOM = (-1)^{k-1} \dfrac{ t_{k+1, k} \, |\zeta_k| }{ u_{k,k} \, u_{k+1,k+1} } \quad (k \geq 1).
  \]
\end{theorem}

\begin{proof}
  Recall that $g_k = \nabla q(x_k) = A x_k - b = -r_k$ for all \(k \geq 0\).
  For $k=0$, $p_1\DIOM = v_1 / u_{1,1} = r_0 / (\beta u_{1,1}) = -g_0 / (\beta u_{1,1})$, and therefore $g_0^T p_1\DIOM = -\beta / u_{1,1}$.

  Consider now $k\ge 1$.
  As in the proof of~\Cref{prop:diom_res},
  \[
    r_k = -t_{k+1, k}\,\frac{ \zeta_k }{ u_{k,k} } \, v_{k+1}.
  \]
  In view of \Cref{lem:zetak} and~\eqref{eq:p-orthogonal},
  \begin{align*}
    g_k^T p_{k+1}\DIOM
    = - \frac{ r_k^T v_{k+1} }{ u_{k,k} }
    = t_{k+1, k} \, \frac{ \zeta_k }{ u_{k,k} \, u_{k+1, k+1} }
    = t_{k+1, k}\,\frac{ (-1)^{k-1} |\zeta_k| }{ u_{k,k} \, u_{k+1, k+1}}.
  \end{align*}
\end{proof}

\Cref{rkpk} states that $g_k^T p_{k+1}\DIOM$ is not always positive, which means that $p_{k+1}\DIOM$ is not a descent direction for every $k$.
The following theorem recovers the CG direction from DIOM($m$) vectors and coefficients.

\begin{theorem}\label{alpha_k}
  For $k=1, 2, \ldots, j$, let \(P_k^T A P_k\) be SPD.
  Then,
  \begin{equation}%
    \label{eq:star}
    d_k\PCG = \zeta_{k+1} \, u_{k+1,k+1} \, p_{k+1}\DIOM
    \quad \text{and}\quad \alpha_k\PCG = 1 / u_{k+1,k+1},
    \quad k = 0, \ldots, j-1.
  \end{equation}
\end{theorem}

\begin{proof}
  We know that $V_{k+1}^T A V_{k+1} = T_{k+1} = T_{k+1}^T \in \R^{(k+1)\times (k+1)}$ is tridiagonal.
  The elements of $T_{k+1}$ can be expressed in terms of the CG coefficients~\cite[\S\(6.7\)]{saad2003} as
  \[
    T_{k+1} \;=\;
    \begin{pmatrix}
      \frac{1}{\alpha_0} & \frac{\sqrt{\beta_0}}{\alpha_0} & & \\
      \frac{\sqrt{\beta_0}}{\alpha_0} & \frac{1}{\alpha_1}+\frac{\beta_0}{\alpha_0} & \ddots & \\
      & \ddots & \ddots & \frac{\sqrt{\beta_{k-1}}}{\alpha_{k-1}} \\
      & & \frac{\sqrt{\beta_{k-1}}}{\alpha_{k-1}} & \frac{1}{\alpha_k}+\frac{\beta_{k-1}}{\alpha_{k-1}}
    \end{pmatrix},
    \label{eq:Talphabeta}
  \]
  from which we dropped the superscripts \({}\PCG\) for readability.

  The LU factorization $T_{k+1} = L_{k+1} U_{k+1}$ with \emph{unit} lower bidiagonal $L_{k+1}$ and upper bidiagonal $U_{k+1}$.
  Matching the sub/super-diagonal entries yields
  \begin{equation}\label{eq:u_jl_j}
    \ell_{i+1,i}=\sqrt{\beta_{i-1}}, \qquad
    u_{i,i+1} = \sqrt{\beta_{i-1}} / \alpha_{i-1}, \qquad
    i = 1, \ldots, k.
  \end{equation}
  The diagonal entries satisfy, $u_{1,1} = 1 / \alpha_0$ and
  \[
    u_{j+1,j+1} + \ell_{j+1,j} \, u_{j,j+1}
    = \frac{1}{\alpha_{j}}+\frac{\beta_{j-1}}{\alpha_{j-1}}, \qquad
    j = 1, \ldots, k.
  \]
  Using~\eqref{eq:u_jl_j}, we obtain \(u_{j+1,j+1} = 1 / \alpha_{j}\) for \(j=1, \ldots, k\).

  The second part of the proof results from the above by equating the iterate updates in \Cref{Algo:PCG} and \Cref{Algo:DIOM} by virtue of \Cref{prop:diom=cg}.
  Specifically, for \(k = 0, \ldots, j-1\),
  \begin{align*}
    x_{k+1}\PCG = x_k\PCG + \alpha_k\PCG d_k\PCG = x_k\DIOM + \zeta_{k+1} p_{k+1}\DIOM.
  \end{align*}
\end{proof}

In CG, negative curvature can be detected by examining the sign of $\alpha_k\PCG$.
\Cref{alpha_k} thus implies that negative curvature can be detected in DIOM(\(m\)) by examining the sign of \(u_{k+1,k+1}\).
A strategy similar to that of the truncated conjugate gradient of \citep{steihaug-1983} can then be used with LBFGS(\(m\)) and DIOM(\(m\)) to compute trust-region steps.
We examine that strategy in the next section.

\section{Truncated methods for nonconvex optimization}

We return to the nonlinear, and possibly nonconvex problem~\eqref{P}, and focus on the use of LBFGS(\(m\)) and DIOM($m$) within a trust-region framework.
We wish to determine empirically whether they have advantages over PCG in floating-point arithmetic.

At iteration \(j\) of a typical Newton trust-region method for~\eqref{P}, we constuct a quadratic model~\eqref{quadratic} of the objective about the current iterate $z_j$ where \(c = f(z_j)\), \(b = -\nabla f (z_j)\), and \(A = \nabla^2 f (z_j)\), or possibly an approximation to the Hessian is the latter is not available to too costly to obtain.
The quadratic model is then approximately minimized while ensuring that \(x\) satisfies the trust-region constraint \(\|x\|_2 \leq \Delta_j\) where \(\Delta_j > 0\) is the trust-region radius.
For reference, we state the subproblem as
\begin{equation}\label{eq:tr-subproblem}
  \minimize{x \in \R^n} \ f(z_j) + \nabla f(z_j)^T x + \tfrac{1}{2} x^T A_j x
  \quad \st \quad
  \|x\|_2 \leq \Delta_j,
\end{equation}
where \(A_j = A_j^T \approx \nabla^2 f(z_j)\).
Once a step \(x_j\) has been obtained, the decrease in~\eqref{quadratic} is compared the actual decrease in \(f\) to accept or reject the step, and \(\Delta_j\) is updated accordingly.
\Cref{Algo:TR} summarizes the main features of the procedure.
We refer the interested reader to the comprehensive treatment \citep{conn-gould-toint-2000} for complete details.

\begin{algorithm}[ht]
  \caption{Newton trust-region method}
  \label{Algo:TR}
  \begin{algorithmic}[1]
    \State Choose $z_0 \in \R^n$, \(\Delta_0 > 0\), \(\varepsilon > 0\), \(0 < \eta_1 < \eta_2 < 1\).
    \State $j = 0$
    \While{$\|\nabla f(z_j)\| > \varepsilon$}
    \State%
    \label{stp:compute-tr-step}%
    Compute \(x_j\) as an approximate minimizer of~\eqref{eq:tr-subproblem}.
    \State $\rho = \dfrac{ f(z_j) - f(z_j + x_j) }{ q(0) - q(x_j) }$ 
    \If{$\rho < \eta_1$}
    $\Delta \gets \Delta/4$ \Comment{poor adequacy: reduce the radius}
    \ElsIf{$\rho \ge \eta_2$}
    $\Delta \gets 2\Delta$ \Comment{very good adequacy: increase the radius}
    \EndIf
    \If{$\rho \ge \eta_1$}
    $z_{j+1} = z_j + x_j$ \Comment{accept step if adequacy is sufficient}
    \Else \
    \(z_{j+1} = z_j\) \Comment{reject step if adequacy is insufficient}
    \EndIf
    \State $j \gets j+1$
    \EndWhile
    \State \Return $z_j$
  \end{algorithmic}
\end{algorithm}

Importantly, on Line~\ref{stp:compute-tr-step} of \Cref{Algo:TR}, \(x_j\) need not be an exact solution of~\eqref{eq:tr-subproblem}, but merely induce \emph{sufficient decrease} in the model while remaining in the trust region---see \citep{conn-gould-toint-2000} for what sufficient decrease means.
For our purposes, suffice it to say that the truncated conjugate gradient method of \citet{steihaug-1983} achieves sufficient decrease at the very first iteration, and is one of the most widely-used methods to compute a step in trust-region methods when the trust region is defined in the Euclidean, or in an elliptic, norm.

Because we seek to compute a step from \(z_j\), which lies at the center of the trust-region, it is common to initialize an iterative method from \(x_0 = 0\).
\citet{steihaug-1983} establishes that, as long as positive curvature directions are generated, i.e., \((d_k\PCG)^T A_j d_k\PCG > 0\), the model decreases monotonically along the path connecting the iterates $\{x_k\}_{k \geq 0}$, and that the iterates move further away the center of the trust region at each iteration.
More precisely, $q_j(x_{k+1}) \leq q(x_k + \tau d_k\PCG)$ for all $\tau \in [0,1]$ and $\|x_{k+1}\|_2 \geq \|x_k\|_2$ provided that $x_0 = 0$.
Thus as long as positive curvature directions are generated, either we find a minimum of \(q\) inside the trust region in at most \(n\) iterations, or there is an iteration \(k\) such that \(\|x_k\|_2 < \Delta_j \leq \|x_{k+1}\|_2\).
Because of the monotonicity of \(q\), we can simply stop at the point where the segment connecting \(x_k\) to \(x_{k+1}\) meets the trust-region boundary.
If nonpositive curvature should be detected at iteration \(k\), i.e., \((d_k\PCG)^T A_j d_k\PCG \leq 0\), we can see that \(q\) continues to decrease along \(d_k\PCG\), and we can safely follow that direction until we hit the trust-region boundary.

By the results of the previous sections, those properties continue to hold for all methods of the Broyden class, to LBFGS(\(m\)), and to DIOM(\(m\)) applied to~\eqref{eq:tr-subproblem}.
In PCG and LBFGS(\(m\)), nonpositive curvature is detected by checking the sign of \(\alpha_k\), and in DIOM(\(m\)), that of \(u_{k+1,k+1}\).
\Cref{Algo:LS-TRNC} provides the procedure for LBFGS($m$), and \Cref{Algo:TR-DIOM} for DIOM($m$).


\begin{algorithm}[ht]
  \caption{LBFGS($m$) for trust-region subproblems}
  \label{Algo:LS-TRNC}
  \begin{algorithmic}[1]
    \State Set $A \gets \nabla^2 f(z_j)$, $b \gets -\nabla f(z_j)$, $\Delta_j > 0$, $H_0 \gets I$, $x_0 \gets 0$, $g_0 \gets -b$
    \For{$k = 0,1,\ldots$}
    \State $d_k \gets -H_k g_k$
    \State Compute $\tau > 0$ such that $\|x_k +\tau s_k\|_2 = \Delta_j$ \Comment{step size to the boundary}
    \State $\alpha_k = - g_k^T d_k / (d_k^T A d_k)$
    \If{$d_k^T A d_k \le 0$ or $\alpha_k > \tau$} \Return $x_k + \tau d_k$ \Comment{step to  boundary}
    \EndIf
    \State $s_k \gets \alpha_k d_k$
    \State $x_{k+1} = x_k + s_k$
    \State $y_k \gets \alpha_k A d_k$ \Comment{$y_k = g_{k+1} - g_k = A s_k$}
    \State $g_{k+1} = g_k + y_k$
    \State $H_{k+1} \gets \mathrm{update}(s_k, y_k, H_k)$
    \EndFor
  \end{algorithmic}
\end{algorithm}

\begin{algorithm}[ht]
  \caption{DIOM(\(m\)) for trust-region subproblems}
  \label{Algo:TR-DIOM}
  \begin{algorithmic}[1]
    \State Set $A \gets \nabla^2 f(z_j)$, $b \gets -\nabla f(z_j)$, $\Delta_j > 0$, $x_0\DIOM \gets 0$, $v_1 \gets b / \|b\|$
    \For{$k = 1, 2, \ldots$}
    \State \(m_k = \max(1, \, k - m + 1)\)
    \State Perform Lines~\ref{stp:begin-arnoldi}--\ref{stp:ukk} of \Cref{Algo:DIOM}
    \State $d_k\DIOM = \zeta_k \Big(v_k - \sum_{i = m_k}^{k-1} u_{i,k} p_i\DIOM\Big)$ \Comment{avoid risking a division by zero}
    \State Compute $\tau > 0$ such that $\|x_{k-1}\DIOM +\tau d_k\DIOM\|_2 = \Delta_j$.
    \If{$u_{k,k} \le 0$ or $1 / u_{k,k} > \tau$} \Return $x_{k-1}\DIOM + \tau d_k\DIOM$ \Comment{step to boundary}
    \EndIf
    \State $x_{k}\DIOM = x_{k-1}\DIOM + d_k\DIOM / u_{k,k}$
    \State $\ell_{k+1,k} = t_{k+1,k} / u_{k,k}$ 
    \State $\zeta_{k+1} = -\ell_{k+1,k} \, \zeta_k$ 
    \EndFor
  \end{algorithmic}
\end{algorithm}


We summarize the per-iteration computational requirements of the three methods in \Cref{tab:complexity-cg-diom-lbfgs}.
PCG stores $x_k$, $g_k$, and $d_k\PCG$.
LBFGS($m$) stores $x_k$, $g_k$, and the $m$ most recent pairs $\{s_i, y_i\}$. For LSR1($m$), the update formula is
\[
  H_{k+1}^m = H_k^0 + \sum_{i = k-m+1}^{k} \frac{ z_i z_i^T }{ s_i^T z_i }
  \quad \text{with} \quad z_i = s_i - H_i^m y_i,
\]
so that LSR1($m$) stores $x_k$ and $g_k$, together with the $m$ most recent vectors $z_i$.
DIOM($m$) stores $\{x_k\}$, $m$ recent Krylov basis vectors $v_i$, and $m$ update vectors $p_i\DIOM$.
We do not count storage for scalars.
As expected, PCG has the smallest memory footprint because conjugacy is maintained through short recurrences, at least in exact arithmetic.
Each method performs one operator-vector product with \(A\) per iteration.
For LBFGS($m$), the additional work arises from Strang's two-loop recursion used to apply the inverse-Hessian approximation to a vector, assuming $H_0 = I$ \citep{nocedal-1980}, and from computing $s_k^T y_k$. 
For LSR1($m$), each vector $z_i$ requires an application of the previous limited-memory operator $H_i^m$ to $y_i$.
This product is computed at the time of insertion of \(z_i\) into memory.
Thus, at iteration $k$, only the new product $H_k^m y_k$ is required to form \(z_k\).
This product requires applying the $m$ stored $z_i$ to $y_k$, at a cost of $2mn$ scalar multiplications, and forming $s_k^T z_k$, which requires $n$ scalar multiplications and additions.
Finally, applying $H_{k+1}^m$ to any vector \(v\) requires $2mn$ scalar multiplications, including those needed to compute $z_i^T v$ and $(z_i^T v) z_i$.
The additional work per iteration is therefore $4mn + n$ scalar multiplications.
For DIOM($m$), they arise from the incomplete orthogonalization process on Lines~\ref{stp:begin-arnoldi}--\ref{stp:end-arnoldi}, the LU factorization update on Lines~\ref{stp:begin-lu}--\ref{stp:end-lu} and~\ref{stp:lkp1k}--\ref{stp:zeta-k}, and the update of $p_k\DIOM$ on Line~\ref{stp:pk-diom} of \Cref{Algo:DIOM}.

\begin{table}[ht]
  \centering
  \footnotesize
  \caption{Per-iteration complexity comparison for CG, LBFGS($m$), and DIOM($m$).}
  \label{tab:complexity-cg-diom-lbfgs}
  \begin{tabular}{lcc}
    \hline
    Method & Stored vectors & Additional work per iteration\\
    \hline
    CG & $3$ &  \\
    LBFGS($m$) & $2 + 2 m$ & $4mn + n$ scalar multiplications\\
    LSR1($m$) & $2 + \phantom{2} m$ & $4mn + n$ scalar multiplications\\
    DIOM($m$) & $1 + 2m$ & $3mn + m$ scalar multiplications\\
    \hline
  \end{tabular}
\end{table}

\section{Numerical experiments}

In this section, we first illustrate the practical behavior of PCG, LBFGS($m$), LSR1(\(m\)) and DIOM($m$) on positive-definite linear systems to assess their convergence behavior in finite precision.
The PCG and DIOM($m$) implementations are those of \citep{montoison-orban-2023}, while the LBFGS($m$) and LSR1(\(m\)) implementations are available at \URL{https://github.com/oihanc/Krylov.jl/tree/lsr1}; the corresponding LBFGS and LSR1 operators are implemented in \citep{orban2019linearoperators}.
The DIOM(\(m\)) implementation was modified to include the detection and treatment of nonpositive curvature and trust-region constraint. 


Next, we report results on two nonlinear problems: data assimilation and classification, using the numbers of objective evaluations, gradient evaluations, and Hessian-vector products as performance measures.
We use the trust-region framework implemented in the \texttt{trunk} solver of \citep{migot2026jsosolvers}.

\subsection{Behavior on Positive-Definite Systems}
\phantomsection  

In this section, we examine the behavior of CG, LBFGS, and DIOM for solving linear systems with positive-definite matrices with \(m = 50\) and \(m = n\).
\Cref{tab:pd-matrices} summarizes the test matrices used. In all experiments, we set $x_0 = 0$ and $b = 100[1,\ldots,1]^T$.
\Cref{fig:pd-matrices} shows plots of the relative residual and the quadratic objective value.
We use \(64\)-, \(128\)-, and \(256\)-bit floating-point arithmetic.

\begin{table}[ht]
  \centering
  \footnotesize
  \caption{Positive-definite test matrices used in the experiments.}
  \label{tab:pd-matrices}
  \begin{tabular}{lccc}
    \hline
    & gr\_30\_30 & nos5 & 494\_bus \\
    \hline
    Dimension $n$ & 900 & 468 & 494 \\
    Conditioning $\kappa_2(A)$ & 1.945e+02 & 1.100e+04 & 2.415e+06 \\
    \hline
  \end{tabular}
\end{table}

\begin{figure}[hp]
  \includegraphics[width=\textwidth]{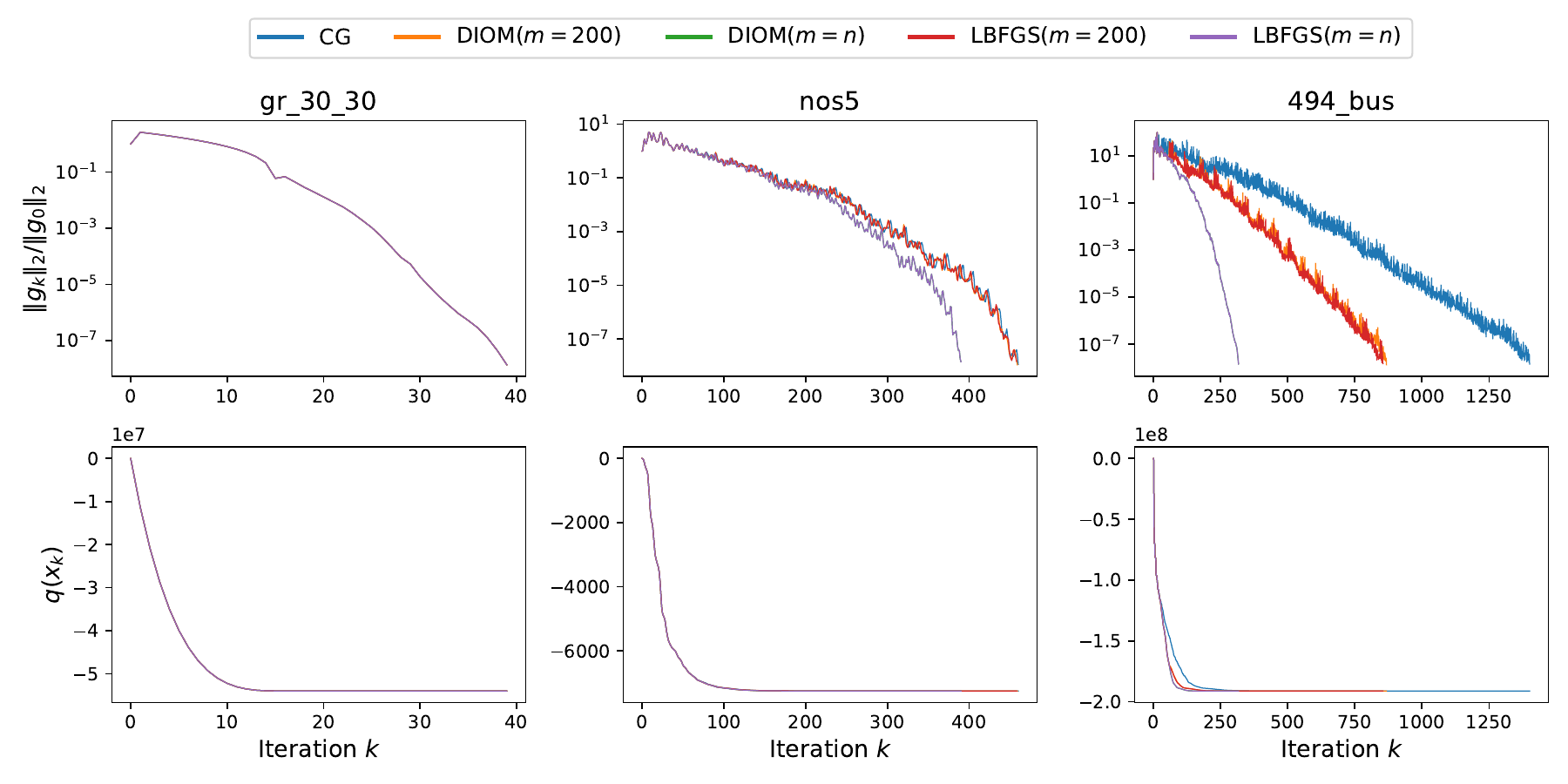}
  \\
  \includegraphics[width=\textwidth]{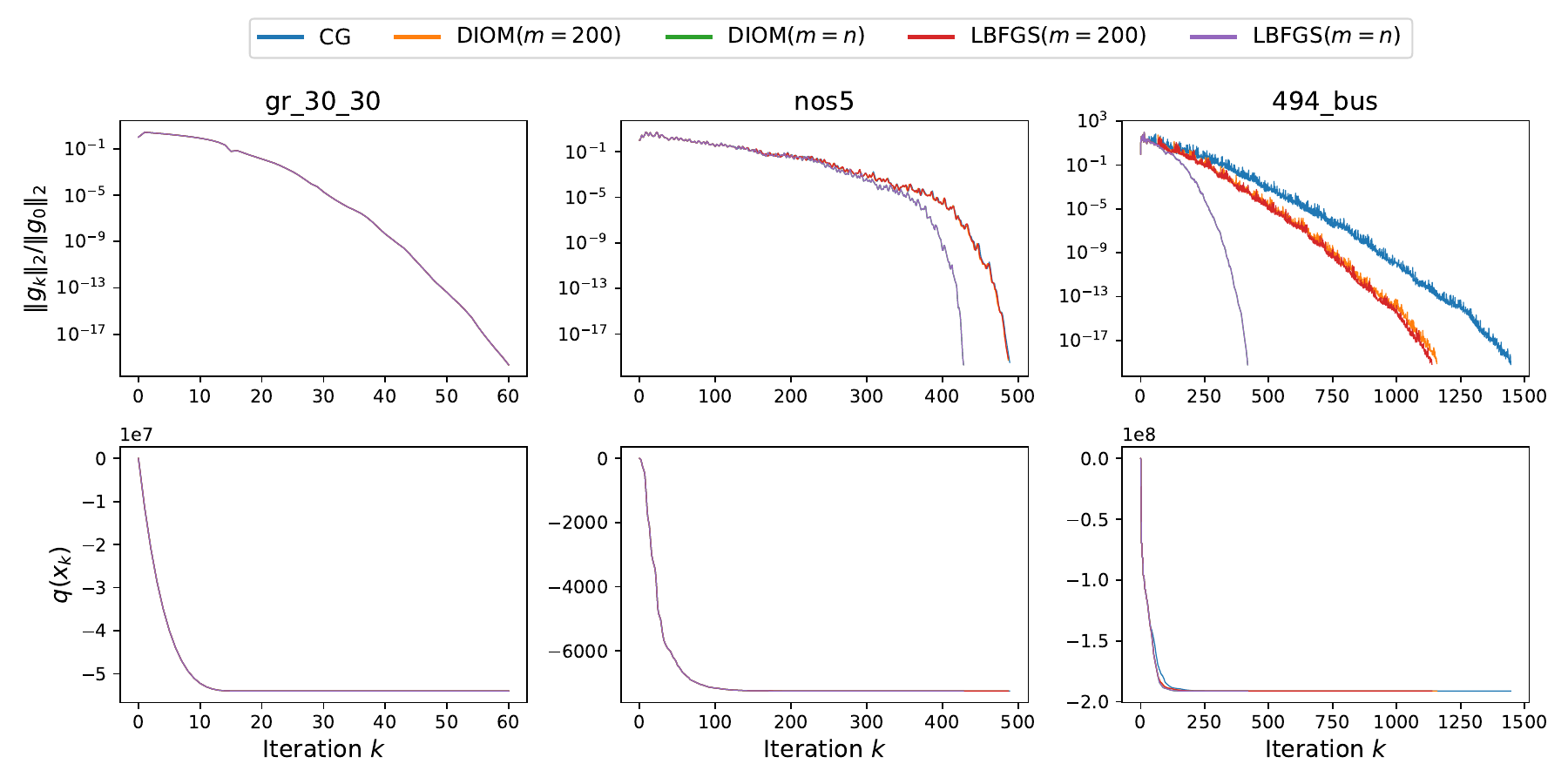}
  \\
  \includegraphics[width=\textwidth]{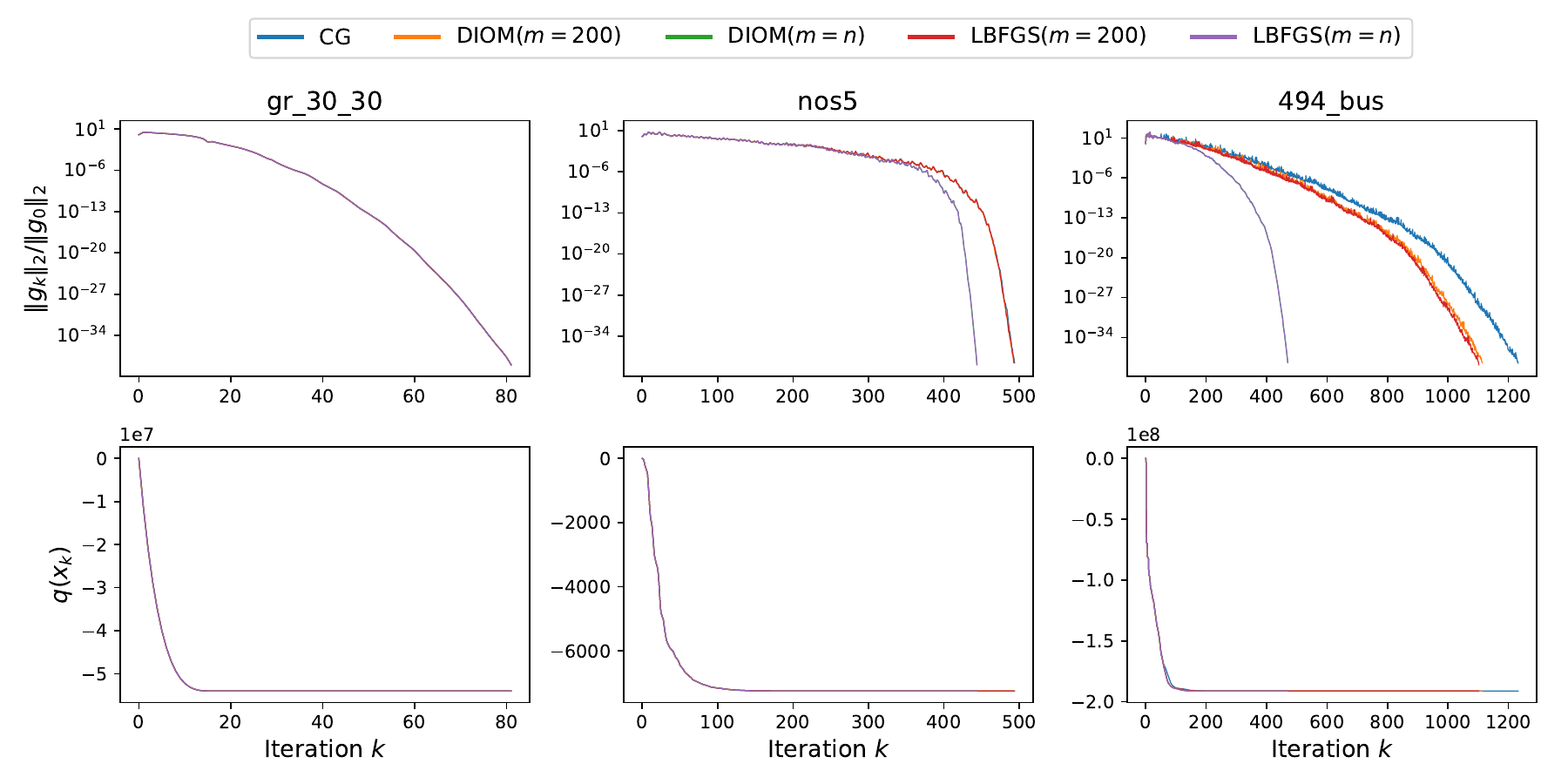}
  \caption{%
    \label{fig:pd-matrices}
    Convergence curves on the test systems of \Cref{tab:pd-matrices} in \(64\)-bit (top), \(128\)-bit (middle), and \(256\)-bit (bottom) floating-point arithmetic.
  }
\end{figure}

These experiments highlight differences in the numerical behavior of the methods in finite precision. In particular, LBFGS and DIOM tend to exhibit more robust convergence behavior than CG on ill-conditioned systems provided \(m\) is sufficiently large.
In the plots, LBFGS(\(m\)) and and DIOM(\(m\)) are visually superposed and perform nearly identically.
Only on the well-conditioned system \texttt{gr\_30\_30} are all the curves superposed.
On the intermediate system \texttt{nos5}, we only see an advantage with DIOM and LBFGS for \(m = n\).
However, even as floating-point accuracy increases, PCG trails behind on the worst-conditioned system \texttt{494\_bus}, even though its condition number can be considered moderate.
This suggests that memory can be an effective mechanism for mitigating the effects of loss of orthogonality on ill-conditioned problems.

The results in~\Cref{fig:pd-matrices-lsr1}, where we plot the same quantities as in~\Cref{fig:pd-matrices} for \(64\)-bit floating-point arithmetic, indicate that full-memory LSR1$(m=n)$ performs well across the different test matrices and is comparable to LBFGS and DIOM with $m=n$.
However, for $m=25, 50, 100, 200$, LSR1 deteriorates once the iteration count exceeds \(m\): even on the well-conditioned system \texttt{gr\_30\_30}, with \(m = 25\), a clear discrepancy between CG and LSR1 appears after about \(25\) iterations.
The same phenomenon occurs for other memory sizes.
This suggests that LSR1 may not generate directions parallel to those of CG, a point that deserves further theoretical investigation.

\begin{figure}[ht]
  \includegraphics[width=\textwidth]{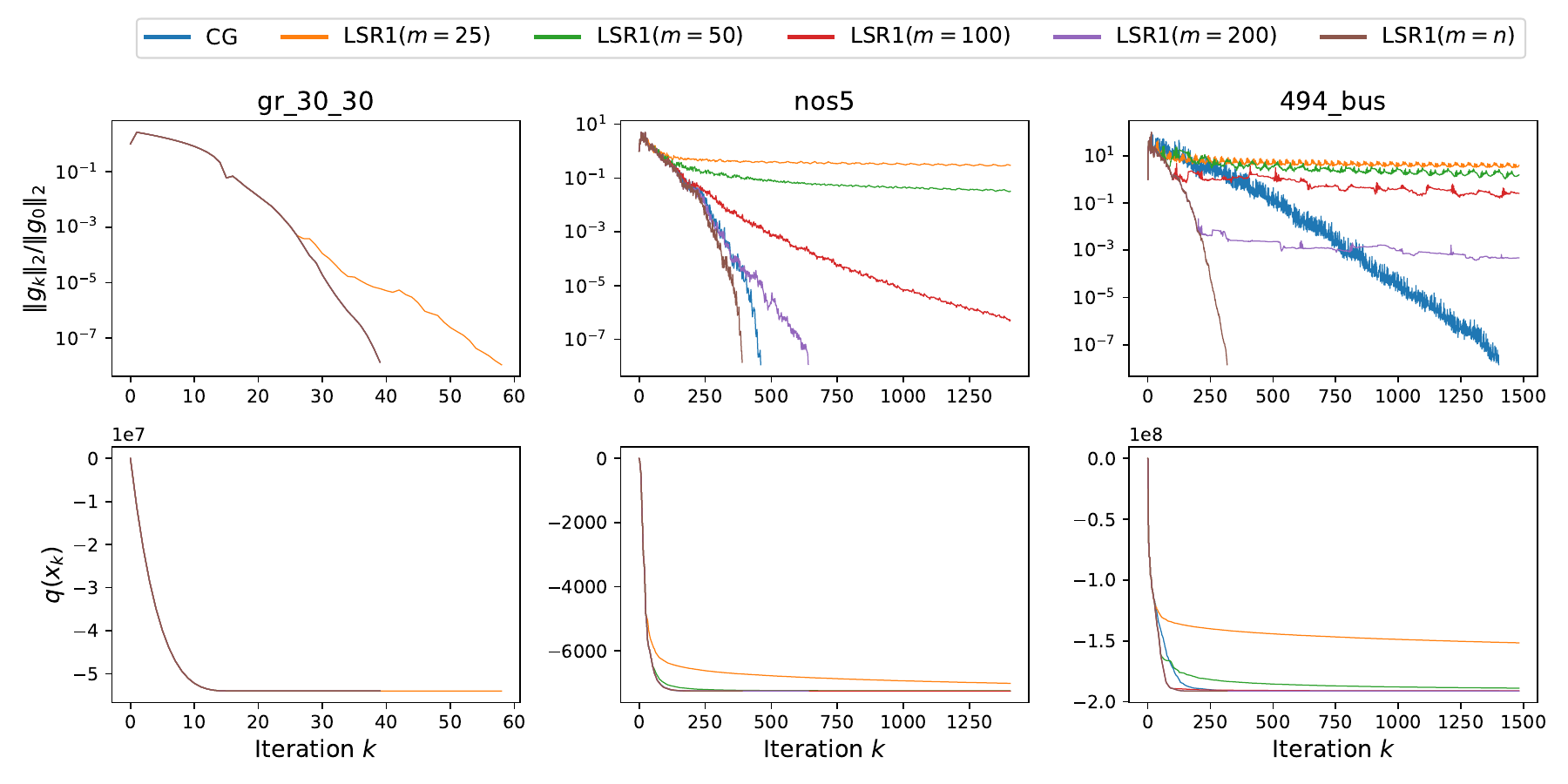}
  \caption{%
    \label{fig:pd-matrices-lsr1}
    Convergence curves for LSR1(\(m\)) and CG on the test systems of \Cref{tab:pd-matrices} in \(64\)-bit floating-point arithmetic.
  }
\end{figure}



\subsection{A data-assimilation problem}%
\label{sec:assimilation}

We evaluate the trust-region framework with CG, LBFGS, and DIOM on a nonlinear weighted least-squares problem arising in data assimilation:
\begin{equation}%
  \label{Prob:nonlinear}
  \minimize{z_0 \in \R^n} f(z_0) = \tfrac{1}{2} \|z_0 - z_b\|_{B^{-1}}^2 + \tfrac{1}{2} \sum_{i=1}^{N_t} \|y_i - \mathcal{H}_i(\mathcal{M}_{t_0, t_i}(z_0))\|_{R_i^{-1}}^2.
\end{equation}
Here, $z_0 = z(t_0)$ denotes the state at the initial time $t_0$, $z_b \in \R^n$ is the background (prior) estimate, and $y_i \in \R^{m_i}$ denotes the observation vector at time $t_i$, for $i = 1, \ldots, N_t$. The nonlinear model $\mathcal{M}_{t_0,t_i}$ propagates $z_0$ from $t_0$ to $t_i$, while $\mathcal{H}_i$ maps the propagated state into the observation space. The SPD covariance matrices $B \in \R^{n \times n}$ and $R_i \in \R^{m_i \times m_i}$ are associated with the background and observation errors, respectively.

In our numerical experiments, we use the Lorenz-96~\citep{lorenz1996predictability} model as the physical dynamical system $\mathcal{M}_{t_0, t_i}$---a common reference model in data assimilation. The observation operator $\mathcal{H}$ is a uniform selection operator, i.e, $\mathcal{H}(x)$ extracts a subset of $x$ that is uniformly selected. We consider
$B = \sigma_b I_n$ with $\sigma_b= 0.8$, $R_1 = R_{2} = \sigma_r^2 I_m$ with $\sigma_r = 0.2$.
We choose $n = 10,000$, $N_t = 2$, and $m_1 = m_2 = 500$. \Cref{tab:assimilation_res} summarizes the number of objective, gradient, and Hessian-vector evaluations, reported as obj\_eval, grad\_eval, and hprod\_eval, respectively.

\begin{table}[ht]
  \centering
  \footnotesize
  \caption{%
    Statistics for the data assimilation problem~\eqref{Prob:nonlinear}.
    Best values are boldfaced.
  }
  \label{tab:assimilation_res}
  \begin{tabular}{l@{\hspace{2pt}}rrrrr}
    \hline
          & $m$   & obj\_eval & grad\_eval & hprod\_eval & time (s) \\
    \hline
    CG    &       & 333 & \textbf{61}   & 1083  & 151.8    \\
    LBFGS & $50$  & \textbf{293} & 63   & \textbf{900}   & \textbf{116.7}    \\
    DIOM  & $50$  & 334 & 68   & 1033  & 128.3    \\
    LBFGS & $100$ & \textbf{293} & \textbf{61}   & 951   & 128.7    \\
    DIOM  & $100$ & 337 & 64   & 1005  & 132.9   \\
    \hline
  \end{tabular}
\end{table}

The results indicate that LBFGS is the most effective subsolver for~\eqref{Prob:nonlinear}, and $m=50$ performs best, requiring the fewest objective evaluations and Hessian-vector products while also achieving the shortest runtime. DIOM also reduces the number of Hessian-vector products relative to CG and yields a shorter runtime.

\subsection{A binary classification problem}
\label{sec:classification}

We benchmark the LBFGS and DIOM subsolvers within the trust-region framework on a binary classification problem derived from the MNIST dataset. The goal is to classify images of the handwritten digits $1$ and $7$. Let $A \in \R^{n \times N}$ denote the data matrix, whose columns are vectorized images, and $b \in \{-1,1\}^N$ denote the corresponding label vector, where $b_i = 1$ for digit $1$ and $b_i = -1$ for digit $7$. Here, $n$ is the number of pixels per image and $N$ is the number of samples. The problem is formulated as
\begin{equation}%
  \label{Prob:classification}
  \minimize{z \in \R^n} \; f(z) = \tfrac{1}{2} \| \mathbf{1} - \tanh(b \odot (A^T z)) \|^2,
\end{equation}
where $\odot$ is the elementwise product, and \(\mathbf{1}\) is the vector of ones. This formulation encourages $b_i (A_i^T z)$ to be positive when sample $i$ is correctly classified. We have $N = 13,007$, and the images have resolution $28 \times 28$, so $n = 784$. \Cref{tab:classification_res} reports our results. Overall, DIOM($m=50$) is the fastest, while requiring the fewest Hessian-vector products. The number of objective and gradient evaluations are identical across all methods, suggesting that DIOM solves the quadratic subproblems more efficiently.

\begin{table}[ht]
  \centering
  \footnotesize
  \caption{%
    Statistics for the data assimilation problem~\eqref{Prob:classification}.
    Best values are boldfaced.
  }
  \label{tab:classification_res}
  \begin{tabular}{l@{\hspace{2pt}}rrrrr}
    \hline
    & $m$        & obj\_eval & grad\_eval & hprod\_eval & time (s) \\
    \hline
    CG    &         & 48 & 26 & 386 & 9.69687 \\
    LBFGS & $50$  & 48 & 26 & 357 & 8.46662 \\
    DIOM  & $50$  & 48 & 26 & \textbf{336} & \textbf{8.46505} \\
    LBFGS & $100$ & 48 & 26 & 357 & 9.36065 \\
    DIOM  & $100$ & 48 & 26 & \textbf{336} & 9.47864 \\
    \hline
  \end{tabular}
\end{table}

\section{Discussion}

While it is well known that partial reorthogonalization in PCG improves its behavior on ill-conditioned systems \citep{horst-1984}, it is interesting that methods that are much closer to optimization, such as LBFGS, also coincide with PCG in exact arithmetic and provide a form of stabilization by way of increased memory.
DIOM can be viewed as indiscriminate partial reorthogonalization, in the sense that it always occurs, and follows naturally from the \citeauthor{arnoldi-1951} process.
It coincides with CG at one end of the spectrum, \(m = 1\), and with FOM at the other, \(m = \infty\).
In exact arithmetic, LBFGS and DIOM coincide with CG for any value of \(m \geq 1\).
The present research and the work of \citep{ek-forsgren-2021} suggest investigating other limited-memory methods of the Broyden class to solve quadratic problems, e.g., in trust-region and line search methods.

\citet{dahito-orban-2019} show that the conjugate residuals method \citep{hestenes-stiefel-1952}, and hence MINRES \citep{paige-saunders-1975}, are appropriate to solve trust-region subproblems, can detect nonpositive curvature, and sometimes outperform CG.
MINRES is based on the \citeauthor{lanczos-1950} process, and has a counterpart based on the \citeauthor{arnoldi-1951} process: GMRES \citep{saad-schultz-1986}.
GMRES admits a limited-memory variant DQGMRES \citep{saad-wu-1996}.
We may expect that DQGMRES will provide added robustness by way of memory and reorthogonalization, and we intend to study its properties in future research on ill-conditioned nonlinear problems. 

The numerical behavior of LSR1 is more delicate. In our experiments, full-memory LSR1 is competitive with the other full-memory methods, but the loss of old curvature pairs appears to have a structural effect: once information older than the memory window is discarded, the method may no longer produce directions parallel to those of CG, even on well-conditioned problems.
Understanding whether this deterioration can be explained by the indefiniteness of the LSR1 inverse approximations, by a vanishing LSR1 denominator, or by a loss of conjugacy is an interesting question.

\small
\bibliographystyle{abbrvnat}
\bibliography{abbrv,bfgs-tr}
\normalsize

\vfill

\tableofcontents

\end{document}